\documentclass[12pt,reqno]{amsart}

\usepackage{etex}
\usepackage[utf8]{inputenc}
\usepackage{amsfonts}
\usepackage{amsmath}
\usepackage{amssymb}
\usepackage{amsthm}
\usepackage{mathrsfs}
\usepackage{stmaryrd}
\usepackage{color}
\usepackage[english]{babel}
\usepackage{fontenc}
\usepackage{url}
\usepackage{graphicx}
\usepackage{cancel}
\usepackage{bookmark}
\usepackage{caption}
\usepackage{epstopdf}
\usepackage{hyphenat}
\usepackage{float}
\usepackage{indentfirst}
\usepackage[export]{adjustbox}
\usepackage{tikz}
\usepackage{color}
\usepackage[all]{xy}
\usetikzlibrary{matrix,arrows,decorations.pathmorphing}
\usepackage{booktabs}
\usepackage{subcaption}
\usepackage{array}
\newcolumntype{P}[1]{>{\centering\arraybackslash}p{#1}}

\allowdisplaybreaks

\newtheorem{theorem}{Theorem}[section]
\newtheorem{lemma}[theorem]{Lemma}
\newtheorem{proposition}[theorem]{Proposition}
\newtheorem{corollary}[theorem]{Corollary}

\theoremstyle{definition}
\newtheorem{remark}[theorem]{Remark}
\newtheorem{example}[theorem]{Example}
\newtheorem{definition}[theorem]{Definition}

\tikzset{
    partial ellipse/.style args={#1:#2:#3}{
        insert path={+ (#1:#3) arc (#1:#2:#3)}
    }
}

\renewcommand{\arraystretch}{1.1}

\makeatletter
\def\l@subsection{\@tocline{2}{0pt}{2.5pc}{5pc}{}}
\makeatother

\author{Sandra Di Rocco}
\address{Department of Mathematics, KTH Royal Institute of Technology, SE-100 44 Stockholm, Sweden}
\email{dirocco@math.kth.se}

\author{Lukas Gustafsson}
\address{Department of Mathematics, KTH Royal Institute of Technology, SE-100 44 Stockholm, Sweden}
\email{lukasgu@math.kth.se}

\author{Luca Schaffler}
\address{Dipartimento di Matematica e Fisica, Universit\`a degli Studi Roma Tre, Largo San Leonardo Murialdo 1, 00146, Roma, Italy}
\email{luca.schaffler@uniroma3.it}

\title[Gaussian likelihood geometry of projective varieties]{Gaussian likelihood geometry of projective varieties}

\begin{document}

\maketitle

\begin{abstract}
We explore the maximum likelihood degree of a homogeneous polynomial $F$ on a projective variety $X$, $\mathrm{MLD}_F(X)$, which generalizes the concept of Gaussian maximum likelihood degree. We show that $\mathrm{MLD}_F(X)$ is equal to the count of critical points of a rational function on $X$, and give different geometric characterizations of it via topological Euler characteristic, dual varieties, and Chern classes.
\end{abstract}

\section{Introduction}

Let $\mathbb{S}^n$ denote the space of complex symmetric $n\times n$ matrices. For a choice of generic matrix $S \in\mathbb{S}^n$, consider the \emph{Gaussian log likelihood function}
\begin{equation*}
\ell_{S}(M)=\log(\det(M))-\mathrm{tr}(SM).
\end{equation*}
The number of invertible critical points of $\ell_{S}$, which we count with multiplicity, on the affine cone of a linear space $\mathcal{L}\subseteq\mathbb{P}(\mathbb{S}^n)$ is called the \emph{Gaussian maximum likelihood degree of $\mathcal{L}$} or Gaussian ML degree for short (this is independent from $S$ if it is picked generically). This concept has been intensively studied, see for instance \cite{SU10,MMMSV23,AGKMS21}. We give a more detailed account in Section~\ref{sec:Application-to-Gaussian-ML-degree}. This definition can be generalized by considering instead of a linear space $\mathcal{L}$, any irreducible subvariety $X\subseteq\mathbb{P}(\mathbb{S}^n)$. 

From a different angle, in \cite{CHKS06} the authors provide a geometrical setting for the study of the critical points of $\log f$ for an arbitrary rational function on $X\subseteq\mathbb{P}^n$ smooth and irreducible. They do so via characteristic classes such as Chern classes and topological Euler characteristic. Our objective is to provide a similar toolbox for an appropriate generalization of the Gaussian ML degree. 

Starting from a homogeneous polynomial $F\in\mathbb{C}[x_0,\ldots,x_n]$, we can construct a \textit{log function} by considering
\begin{equation}
\label{eq:log-function-of-reference}
\ell_{F,u}(x):=\log(F(x)) - u(x),
\end{equation}
where $u(x)$ is a generic linear polynomial. We will be concerned with understanding the number of critical points (counted with multiplicities) of $\ell_{F,u}|_{C_X^\circ}$, where $C_X\subseteq\mathbb{C}^{n+1}$ is the affine cone over $X$ and $C_X^\circ$ the open subset of the smooth locus of $C_X$ where $F\neq0$. This class of functions fits the framework presented in \cite{KKS21}. In particular, for a generic choice of $u$, the number of critical points of $\ell_{F,u}|_{C_X^\circ}$ is independent from $u$. We denote this number by $\mathrm{MLD}_F(X)$. Section~\ref{sec:The-Maximum-Likelihood-Degree} provides a self-contained scheme theoretical account of this framework, largely already treated in \cite{KKS21}, for the reader's convenience and to set up appropriate notations. 

This paper provides a number of different characterizations of  $\mathrm{MLD}_F(X)$. Consider an irreducible variety $X = V(g_1, \ldots, g_k) \subseteq \mathbb{P}^n$. Moreover, as introduced previously, consider a generic linear polynomial $u(x)$. We first prove a geometric characterization.

\begin{theorem} 
[Corollary~\ref{Gaussian-geometric-schemes-are-equal}] The number  $\mathrm{MLD}_F(X)$ equals the number of critical points of $\log f_{F,u}$ away from $D_{F,u}:= \mathrm{Supp}(\mathrm{div}(f_{F,u}))$, where $f_{F, u}$ is the rational function on $X$ given by $ F/u^{\deg(F)}$ for a generic linear polynomial $u(x)$.
\end{theorem}

As a corollary, we obtain the following statement for the Gaussian ML degree.

\begin{corollary}
Let $X\subseteq\mathbb{P}(\mathbb{S}^n)$ be an irreducible closed subvariety. Then the Gaussian ML degree of $X$ equals the number of critical points of $\log f_{\det, S}$ away from $D_{\det, S}:= \mathrm{Supp}(\mathrm{div}(f_{\det, S}))$, where $f_{\det, S}(M)$ is the rational function on $X$ given by $\det(M)/\mathrm{tr}(SM)^n$ for a generic $S\in\mathbb{S}^n$.
\end{corollary}

We note that the equality between Gaussian ML degree and the number of critical points of a rational function gives new ways to compute the former. The following theorem collects Theorems~\ref{thm:crit-degree-as-Euler-characteristic}, \ref{thm:MLD-via-dual-variety}, and \ref{thm:Gauss-ML-degree-using-tot-Chern-class} respectively, giving different ways of computing $\mathrm{MLD}_F(X)$.

\begin{theorem}
With assumptions as above the following statements hold:
\begin{enumerate}

\item $\mathrm{MLD}_F(X)=(-1)^{\dim(X)}\chi(\mathrm{Eu}_{X}|_{X\setminus D_{F,u}})$, where $\mathrm{Eu}_X$ is the local Euler obstruction function on $X$. (We have that $\chi(\mathrm{Eu}_{X}|_{X\setminus D_{F,u}})=\chi_{\mathrm{top}}(X\setminus D_{F,u})$ if $X$ is smooth away from $V(F)$).

\item Let $\overline{\mathbb{X}}_F \subseteq \mathbb{P}^{n+1}$ be the projective closure of  $\mathbb{X}_F = V(g_1, \ldots, g_k, F(x) - 1) \subseteq \mathbb{C}^{n+1}$ and let $\overline{\mathbb{X}}^{\vee}_F$ denote its dual variety. Let $m(\overline{\mathbb{X}}_F^\vee)$ be the multiplicity of $\overline{\mathbb{X}}_F^\vee$ at the point $[0:\ldots:0:1]$. Then
\begin{displaymath}
\mathrm{MLD}_F(X) = \left\{ \begin{array}{ll}
\frac{\deg(\overline{\mathbb{X}}_F^{\vee})-m(\overline{\mathbb{X}}_F^\vee)}{\deg(F)} &\textrm{if}~\overline{\mathbb{X}}_F^{\vee}~\text{is a hypersurface},\\
0&\textrm{otherwise}.
\end{array} \right.
\end{displaymath}

\item Let $\pi\colon(Y, B) \to (X, D_{F,u})$ be a log resolution, where $B=\pi^{-1}(D_{F,u})$ and $X$ is assumed to be smooth away from $V(F)$. If $B = \mathrm{Supp}(\mathrm{div}(f_{F, u} \circ \pi)) = \sum_{i=1}^r B_i$, where the latter equality is the decomposition into irreducible components, then $\mathrm{MLD}_F(X)$ is the degree of the coefficient in front of $z^d$ in
\[
c_{\mathrm{tot}}(\Omega_Y) \prod_{i=1}^r (1 - zB_i)^{-1} \in A^*(Y)[z].
\]

\end{enumerate}
\end{theorem}

The equality in Theorem~\ref{thm:crit-degree-as-Euler-characteristic} in the case of the Gaussian ML degree for a linear space $\mathcal{L}\subseteq\mathbb{P}(\mathbb{S}^n)$ was already implicit in the literature by combining \cite{DMS21, DP03}. Finally, the formula in Theorem~\ref{thm:Gauss-ML-degree-using-tot-Chern-class} is an application of \cite[Corollary~5]{CHKS06}.

We show how the previous result is applied via a simple running example of a variety that is statistically feasible, in the sense that it contains positive definite matrices. 

\begin{example}
\label{ex:running-example-intro}
Consider the linear space $\mathcal{L}\subseteq\mathbb{P}(\mathbb{S}^3)$ given by the span of $\left[\begin{smallmatrix}1&0&0\\0&0&0\\0&0&0\end{smallmatrix}\right]$, $\left[\begin{smallmatrix}0&1&0\\1&0&0\\0&0&0\end{smallmatrix}\right]$, and $\left[\begin{smallmatrix}0&0&0\\0&1&0\\0&0&1\end{smallmatrix}\right]$. 
It can be interpreted as a linear concentration model of Gaussian probability distributions, appropriate for maximum likelihood estimation. 
In this context we consider $F=\det$ restricted to $\mathcal{L}$. This yields a reducible plane curve $\det\left(\left[\begin{smallmatrix}x&y&0\\y&z&0\\0&0&z\end{smallmatrix}\right]\right)=z(xz-y^2)=0$ in $\mathcal{L}$, pictured on the left of Figure~\ref{fig:running-example}. The reader can find the above three methods applied in Examples~\ref{ex:running-example-topological-Euler-characteristic}, \ref{ex:running-example-dual-variety}, and \ref{ex:running-example-tot-CHern-class}. These methods show that $\mathrm{MLD}_{\det}(\mathcal{L})=1$. The second method uses the cubic surface $\mathbb{X}_{\det}=V(z(xz-y^2)-1)\subseteq\mathbb{C}^3$ on the right of Figure~\ref{fig:running-example}. In Example~\ref{ex:running-example-Milnor} we also compute $\mathrm{MLD}_{\det}(\mathcal{L})$ using Milnor numbers.
\end{example}

\begin{figure}[H]
\centering
\begin{subfigure}{.3\textwidth}
\begin{tikzpicture}[scale=0.6]

	\draw[line width=1.5pt] (-4,0) -- (4,0);
	
	\draw[line width=1.5pt] (0,3) [partial ellipse=0:360:2cm and 3cm];

\end{tikzpicture}
\end{subfigure}
\hspace{0.5in}
\begin{subfigure}{.4\textwidth}
  \centering
  \includegraphics[width=\linewidth]{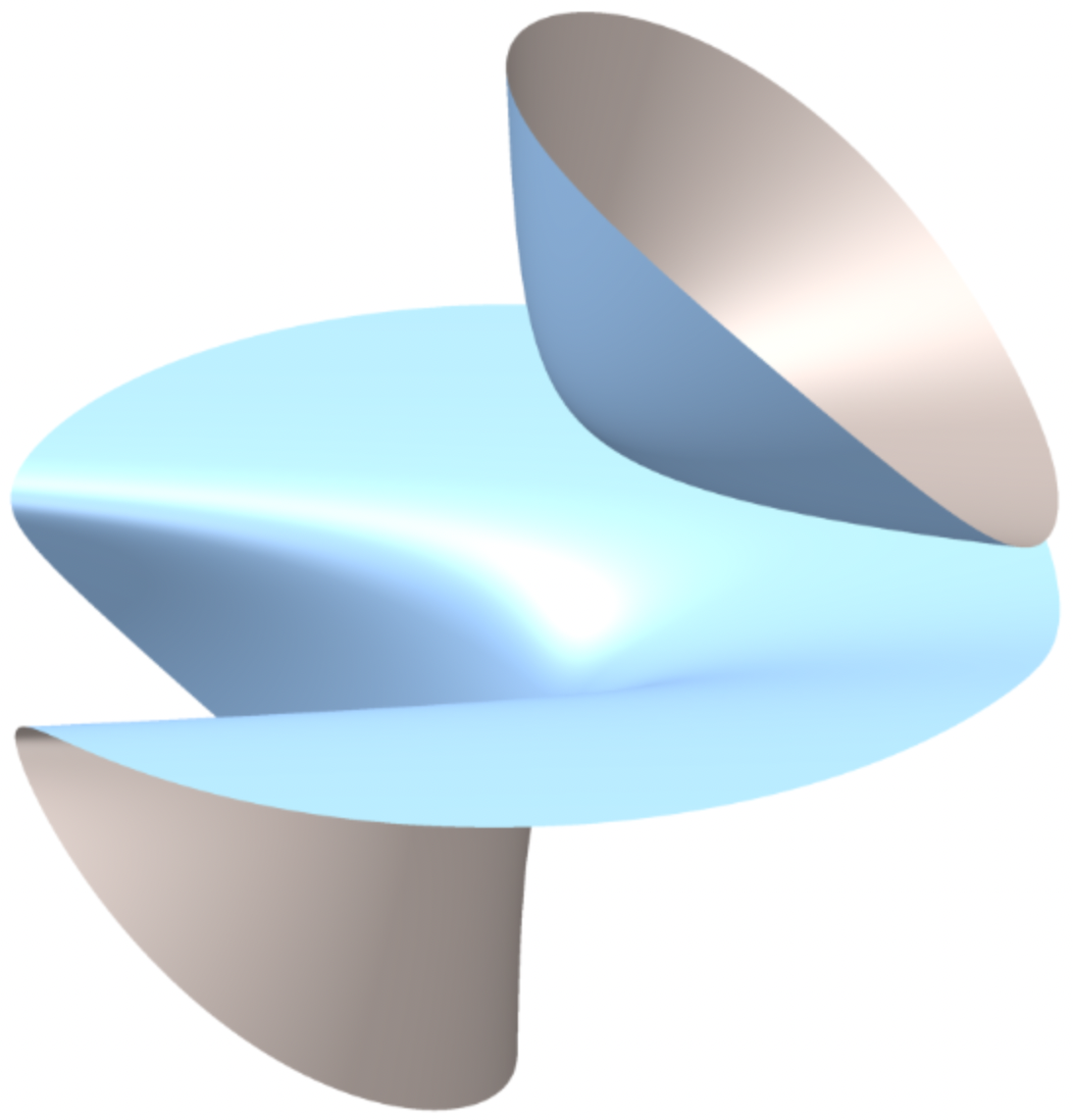}
\end{subfigure}
\caption{With reference to Example~\ref{ex:running-example-intro}, on the left the curve $V(F)$ on $\mathcal{L}$. On the right, the affine variety $\mathbb{X}_{\det}$ as defined in Theorem~\ref{thm:MLD-via-dual-variety} (made with SURFER).}
\label{fig:running-example}
\end{figure}
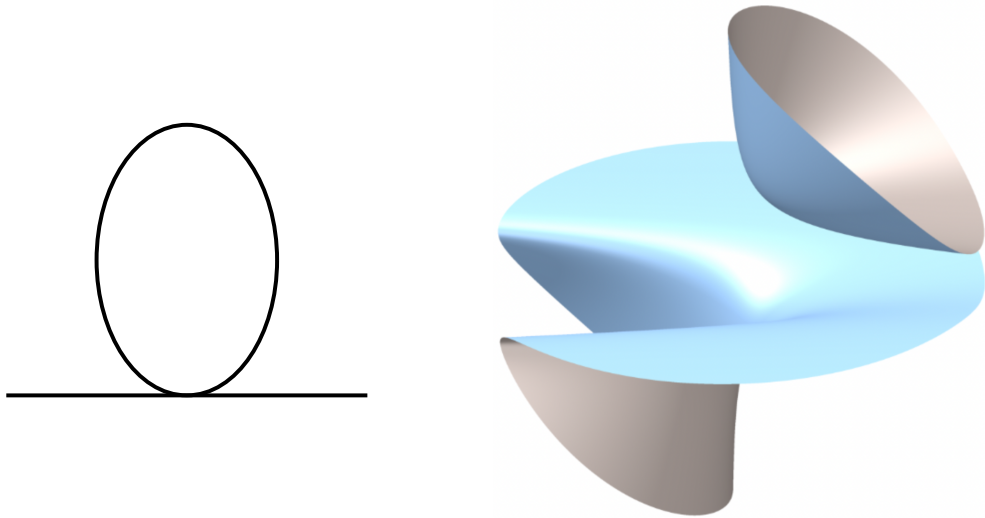

The developed theory relates well and gives some advances to the two most commonly explored concepts of ML degree: The Gaussian, as we already discussed, and the {\it discrete} ML degree.
In Section~\ref{sec:Application-to-Gaussian-ML-degree} we illustrate how above results apply to Gaussian maximum likelihood degree. In Section~\ref{discrete} we relate our approach to the {\it discrete maximum likelihood degree} studied in \cite{CHKS06,Huh13}. More precisely we show that:

\begin{theorem}[Theorem~\ref{ineqDiscrete}]
View $\mathbb{P}^n\subseteq\mathbb{P}(\mathbb{S}^n)$ as the projectivized space of diagonal matrices and let $X \subseteq \mathbb{P}^n$ be a closed subvariety. Then
\[
\mathrm{Discrete~ML~degree}(X) \leq \mathrm{Gaussian~ML~degree}(X).
\]
\end{theorem}

\section*{Acknowledgements} We would like to thank Kathl\'en Kohn, Orlando Marigliano, Mateusz Micha\l ek, Felix Rydell, and Bernd Sturmfels for helpful conversations. In particular, we thank Laurentiu G. Maxim for suggesting the extension of our original results to the non-smooth case. We also thank the anonymous referees for the valuable comments and suggestions. The first and second author are supported by the VR grant [NT:2018-03688]. The third author is a member of the INdAM group GNSAGA and was partially supported by the projects ``Programma per Giovani Ricercatori Rita Levi Montalcini'', PRIN2017SSNZAW ``Advances in Moduli Theory and Birational Classification'', PRIN2020KKWT53 ``Curves, Ricci flat varieties and their Interactions'', and, while at KTH, by a KTH grant by the Verg foundation.

\section{The maximum likelihood degree}
\label{sec:The-Maximum-Likelihood-Degree}

In this section we revisit the concept of algebraic degree introduced in \cite{NR09}, and recently explored for a more general class of objective functions in \cite{KKS21}. We take a scheme theoretical approach, that will allow us to establish a bridge to rational functions, later explored in Section~\ref{sec:geometric-formulation}. Additionally, on the way we are able to prove some stronger results about the structure of the critical point locus for the specific class of functions $\ell_{F,u}$.

In what follows, we fix the canonical basis of $\mathbb{C}^{n+1}$, which induces an identification $\mathbb{C}^{n+1}\cong(\mathbb{C}^{n+1})^\vee$ and a bilinear form on $\mathbb{C}^{n+1}$. Therefore, if $u\in\mathbb{C}^{n+1}$ and we want to view it as a functional, then we let $u(x):=u\cdot x$.

Let $F \in \mathbb{C}[x_0, \ldots, x_n]$ be a homogeneous polynomial and let $I(X)=(g_1,\ldots,g_k)$ be the homogeneous ideal in $\mathbb{C}[x_0, \ldots, x_n]$ defining $X\subseteq\mathbb{P}^n$ irreducible and of codimension $c$. We let $C_X^\circ$ be the open subset of the smooth locus of $C_X$ where $F\neq0$. Note that we have an open immersion
\[
C_X^\circ \subseteq \mathrm{Spec}(\mathbb{C}[x_0, \ldots, x_n, 1/F]/I(X)).
\]

\begin{definition}
\label{def:principal-open-from-minor-of-Jacobian-matrix}
With the notation established above, we have that $C_X^\circ$ is equipped with a covering of principal open subsets, denoted by $U_\mu$, of $\mathrm{Spec}(\mathbb{C}[x_0, \ldots, x_n, 1/F]/I(X))$ where some $c\times c$ minor $\mu$ of the Jacobian matrix of $g_1, \ldots, g_k$ is nonzero. That is,
\[
U_\mu:=\mathrm{Spec}(\mathbb{C}[x_0, \ldots, x_n, 1/F, 1/\mu]/I(X)) \subseteq C_X^\circ.
\]
\end{definition}

If $M=(m_{ij})$ is an $m\times n$ matrix over a ring $R$ and $k\leq \min\{m,n\}$, we define $I_k(M)$ to be the ideal of $R$ generated by the $k\times k$ minors of $M$.

\begin{definition}
\label{def:A(x,eta)-and-critical-points-ell}
Let $A(x,u)=A_{F,X}(x,u)$ be the matrix given by
\[
A(x,u)=\begin{bmatrix}
\frac{\nabla F}{F} - u \\
\nabla g_1\\
\vdots\\
\nabla g_k \\
\end{bmatrix} \in \mathrm{Mat}_{(k+1) \times (n+1)}(\mathbb{C}[x_0, \ldots, x_n, 1/F]).
\]
\end{definition}

It follows from the theory of Lagrange multipliers that the scheme of critical points of $\ell_{F, u}|_{C_X^\circ}$ on the principal open set $U_\mu$ is cut out by the following ideal:
\[
I_{c+1}(A(x,u))  \subseteq \mathbb{C}[x_0, \ldots, x_n, 1/F, 1/\mu]/I(X) = \mathbb{C}[U_\mu].
\]
We point out that, as a consequence of the above setting, the critical points are taken to be smooth points of $X$ by definition.

\begin{remark}
\label{rem:critical-fiber}
Let $\mathcal{I}\subseteq C_X^\circ \times \mathbb{C}^{n+1}$ be the closed subvariety defined by $I_{c+1}(A(x,u))$, but with the entries of $A(x,u)$ viewed as elements of $\mathbb{C}[x_0, \ldots, x_n ,1/F, u_0, \ldots, u_n]$. We denote by $\pi_1,\pi_2$ the restrictions to $\mathcal{I}$ of the projections to $C_X^\circ$ and $\mathbb{C}^{n+1}$ respectively. We have that the image of $\pi_1$ equals $C_X^\circ$. Summarizing, we have the following diagram
\begin{center}
\begin{tikzpicture}[>=angle 90]
\matrix(a)[matrix of math nodes,
row sep=2em, column sep=2em,
text height=1.5ex, text depth=0.25ex]
{&\mathcal{I}&\\
C_X^\circ&&\mathbb{C}^{n+1}.\\};
\path[->] (a-1-2) edge node[above left]{$\pi_1$}(a-2-1);
\path[->] (a-1-2) edge node[above right]{$\pi_2$}(a-2-3);
\end{tikzpicture}
\end{center}
Proposition~\ref{lem:reduced-crit-pts} provides the link between $\pi_2$ being dominant and the log likelihood behavior. We observe that for $u\in\mathbb{C}^{n+1}$, the scheme-theoretic fiber $\pi_2^{-1}(u)$ is the scheme of critical points of $\ell_{F,u}|_{C_X^\circ}$. 
\end{remark}

The following lemma will be used several times throughout the paper.

\begin{lemma}
\label{technical-lemma-Cramer-rule}
Let $A$ be an $m \times n$ matrix with entries being elements in $k[Y]$ for some closed subvariety $Y \subseteq \mathbb{A}^k$ such that all $(c+1) \times (c+1)$ minors vanish for some $c < \min\{m,n\}$. Let $\alpha \subseteq \{1, \ldots, m \}$ be a collection of $c$ rows of $A$ for which there exists a nonzero $c \times c$ minor $\mu$. Then every row of $A$ can be written as a unique linear combination with coefficients in $k[Y]_\mu$ of the rows $A_i$, $i \in \alpha$ (here $k[Y]_\mu$ denotes the localization of $k[Y]$ by the element $\mu$).
\end{lemma}

\begin{proof}
Without loss of generality, we can reorder the rows and columns of $A$ so that $\mu$ is the determinant of the submatrix $A_{ij}$ with $1\leq i,j \leq c$. For $j \in \{c+1, \ldots, m \}$, let $A_j$ denote the $j$-th row of $A$. We want to find a solution in $\lambda_1,\ldots,\lambda_c$ in to the following linear system:
\begin{equation}
\label{eq:system-non-augmented}
\begin{bmatrix}
A_{1}^\intercal & \ldots & A_{c}^\intercal \\
\end{bmatrix}
\begin{bmatrix}
\lambda_1 \\ \vdots \\ \lambda_c
\end{bmatrix}
=
A_{j}^\intercal.
\end{equation}
For this purpose, we introduce the augmented matrix of coefficients $[A_1^\intercal\ldots A_c^\intercal|R]$, where $R=\begin{bmatrix}\mathbf{0}\\\mathrm{I}_{n-c}\end{bmatrix}$ and $\mathrm{I}_{n-c}$ is the identity matrix of size $n-c$. We aim to solve the following larger system of equations:
\begin{equation}
\label{eq:system-augmented}
[A_1^\intercal\ldots A_c^\intercal|R]
\begin{bmatrix}
\lambda_1 \\ \vdots \\ \lambda_n
\end{bmatrix}
=
A_{j}^\intercal,
\end{equation}
where $\lambda_{c+1},\ldots,\lambda_n$ are additional indeterminates. In this new linear system we can use Cramer's rule to determine $\lambda_1,\ldots,\lambda_n$. Let $i\in\{1,\ldots,c\}$. Then we have that
\[
\lambda_i=\frac{\det([A_1^\intercal\ldots A_j^\intercal\ldots A_c^\intercal|R])}{\det([A_1^\intercal\ldots A_c^\intercal|R])}=\frac{\det([A_1^\intercal\ldots A_j^\intercal\ldots A_c^\intercal|R])}{\mu}\in k[Y]_\mu,
\]
where $A_j^\intercal$ appears in the $i$-th column. For $i>c$, let us assume without loss of generality that $i=c+1$ (an analogous argument applies to the other cases). Define $R'$ to be $R$ with the first column dropped. Cramer's rule again yields
\[
\lambda_i=\frac{\det([A_1^\intercal\ldots A_c^\intercal A_j^\intercal\;|\; R'])}{\det([A_1^\intercal\ldots A_c^\intercal|R])}=\frac{0}{\mu}=0,
\]
because all the $(c+1)\times(c+1)$ minors of $A$ vanish by assumption. This shows that a solution $(\lambda_1,\ldots,\lambda_c)$ to system \eqref{eq:system-non-augmented} exists. Note that this is also unique as $(\lambda_1,\ldots,\lambda_n)$ is the only solution to \eqref{eq:system-augmented}, because $[A_1^\intercal,\ldots,A_c^\intercal\mid R]$ is invertible and $\lambda_{c+1}=\ldots=\lambda_n=0$. (We point out that an analogous proof works in any reduced commutative ring in place of a coordinate ring $k[Y]$.)
\end{proof}

\begin{lemma}
\label{incidence-is-vec-bund-onCXo}
The projection $\pi_1\colon\mathcal{I}\rightarrow C_X^\circ$ as in Remark~\ref{rem:critical-fiber} is a vector bundle of rank equal to $c=\mathrm{codim}_{\mathbb{P}^n}(X)$.
\end{lemma}

\begin{proof}
Let $U_{\mu}$ be the principal open subset in Definition~\ref{def:principal-open-from-minor-of-Jacobian-matrix}, associated to the rows $\alpha=\{\alpha_1<\ldots<\alpha_c\}$. We now show that $\pi_1^{-1}(U_{\mu})\cong U_{\mu}\times \mathbb{A}^c$. For this we will construct algebraic maps inverse to each other. We start by defining:
\begin{align*}
\Phi\colon U_\mu \times\mathbb{A}^c&\rightarrow\pi_1^{-1}(U_\mu)\\
(x,(\lambda_1,\ldots,\lambda_c))&\mapsto\left(x,\left(-\sum_{j=1}^c\lambda_{j}\nabla g_{\alpha_j}(x)\right)+\frac{\nabla F(x)}{F(x)}\right).
\end{align*}
Note that the function $\Phi$ has image in $\mathcal{I}$ because the $(c+1)\times (c+1)$ minors of the matrix
\[
A\left(x, -\sum_{j=1}^c\lambda_{j}\nabla g_{\alpha_j}(x)+\frac{\nabla F(x)}{F(x)}\right)
\]
(see Definition~\ref{def:A(x,eta)-and-critical-points-ell}) are all zero. This is because a $(c+1)\times (c+1)$ minor not involving the first row is zero by smoothness of $C_X^\circ$, and if it involves the first row then it is zero by Lemma~\ref{technical-lemma-Cramer-rule} applied to the Jacobian matrix of $g_1,\ldots,g_k$.

We construct the inverse of $\Phi$. By Lemma~\ref{technical-lemma-Cramer-rule} and by the fact that $\mu$ never vanishes on $U_\mu$, every row of $A$ can be written as a linear combination of the rows $\nabla g_{\alpha_1},\ldots,\nabla g_{\alpha_c}$ with coefficients in $\mathbb{C}[\pi^{-1}(U_{\mu})]$. So the first row of $A$, namely $\frac{\nabla F}{F}- u$ can be written uniquely as $\sum_{i=1}^c\lambda_i(x,u) \nabla g_{\alpha_i}$ with $\lambda_i(x,u)\in \mathbb{C}[\pi^{-1}(U_\mu)]$. This allows us to define a map:

\begin{align*}
\Psi\colon\pi_1^{-1}(U_\mu)&\rightarrow U_{\mu}\times\mathbb{A}^c\\
(x,u)&\mapsto(x,\lambda_1(x,u),\ldots,\lambda_c(x,u)).
\end{align*}
We now check that $\Phi$ and $\Psi$ are inverse to each other. The composition $\Phi\circ\Psi$ equals $\mathrm{id}_{\pi_1^{-1}(U_{\mu})}$ by the defining property of $\lambda_1(x,u),\ldots,\lambda_c(x,u)$. So let us consider $\Psi\circ\Phi$. We have that

\[
\Psi(\Phi(x,\lambda_1,\ldots,\lambda_c))=\Psi\left(x,\left(-\sum_{i=1}^c\lambda_i\nabla g_{\alpha_i}(x)\right)+\frac{\nabla F(x)}{F(x)}\right)=(x,\ell_1, \ldots , \ell_c),
\]
where $\ell_1,\ldots,\ell_c$ by construction is a collection of functions satisfying,
\[
\sum_{i=1}^c\ell_i\nabla g_{\alpha_i}(x)=\sum_{i=1}^c\lambda_i\nabla g_{\alpha_i}(x).
\]
This combined with the fact that $\mu(x)\neq0$ forces $\ell_i=\lambda_i$ for all $i=1,\ldots,c$.

To conclude, one has to show that on each fiber $\pi_1^{-1}(x)$, $x\in U_\mu$, the map $\Psi$ induces a linear isomorphism with $\mathbb{A}^c$, and for this first we need to clarify the vector space structure on $\pi_1^{-1}(x)$. We set $(x,u)+(x,v)=(x,u+v-\nabla F(x)/F(x))$ and $t\cdot(x,u)=(x,tu+(1-t)\nabla F(x)/F(x))$. These define a vector space structure on $\pi_1^{-1}(x)$ on which $\Psi$ induces the required linear isomorphism.
\end{proof}

\begin{lemma}
\label{lem:reduced-crit-pts}
For a generic choice of $u\in\mathbb{C}^{n+1}$, the scheme of critical points of $\ell_{F, u}|_{C_X^\circ}$ is either empty, or reduced and zero-dimensional. If it is non-empty, the number of critical points is independent of the choice of general $u$.
\end{lemma}

\begin{proof}
The scheme of critical points is a general fiber of the rational map $\pi_2\colon\mathcal{I}\rightarrow \mathbb{C}^{n+1}$ as discussed in Remark~\ref{rem:critical-fiber}. If $\pi_2$ is not dominant, then the general fiber is empty. So assume it is dominant. Then the general fiber of $\pi_2$ has dimension equal to
\[
\dim(\mathcal{I})-\dim(\mathbb{C}^{n+1})=\dim(C_X^\circ)+\mathrm{codim}_{\mathbb{C}^{n+1}}(C_X^\circ)-\dim(\mathbb{C}^{n+1})=0,
\]
where for the first equality we used Lemma~\ref{incidence-is-vec-bund-onCXo}. Finally, let $U_0\subseteq\mathbb{C}^{n+1}$ be the dense open subset such that $\pi_2^{-1}(U_0)\rightarrow U_0$ has zero-dimensional fibers. Let $U_1\subseteq\mathbb{C}^{n+1}$ be the dense open subset such that $\pi_2^{-1}(U_1)\rightarrow U_1$ is a smooth morphism (see \cite[Corollary~10.7]{Har77}). Then, if $U:=U_0\cap U_1$, we have that $\pi_2^{-1}(U)\rightarrow U$ has fibers consisting of finitely many reduced points. This number is the degree of the rational map $\pi_2$.
\end{proof}

\begin{definition}
Let $F,X$ be as at the beginning of Section~\ref{sec:The-Maximum-Likelihood-Degree}. The \emph{maximum likelihood degree of the polynomial $F$ on $X$} is defined as
\[
\mathrm{MLD}_F(X)= \#\,\textrm{critical points of } \ell_{F, u}|_{C_X^\circ},
\]
where ``\#'' denotes the set-theoretic number of critical points and $u$ is chosen generically (so that all the critical points are reduced). This is well defined by Lemma~\ref{lem:reduced-crit-pts}. In the case $X \subseteq \mathbb{P}(\mathbb{S}^n)$, then $\mathrm{MLD}_{\det}(X)$ is the well studied Gaussian ML degree.
\end{definition}

\section{Geometric formulation}
\label{sec:geometric-formulation}

\begin{definition}
Let $F(x)$ be a homogeneous polynomial on $\mathbb{P}^n$ and $u(x)$ a linear form. Define the rational map
\begin{align*}
f_{F, u}\colon X &\dashrightarrow\mathbb{C}, \\
x &\mapsto \mathrm{deg}(F)^{\mathrm{deg}(F)} \frac{F(x)}{u(x)^{\mathrm{deg}(F)}}.
\end{align*}
Observe that this is a well defined rational function. The reason for introducing the coefficient $\mathrm{deg}(F)^{\mathrm{deg}(F)}$ is technical and it will be clear in the proof of Lemma~\ref{lem:affine-patch}. We define $D_{F, u} := \mathrm{Supp}(\mathrm{div}(f_{F, u}))$, the support of the divisor associated to $f_{F,u}$.
\end{definition} 

Following \cite{CHKS06}, we consider the scheme associated to the vanishing of the differential $d\log(f_{F,u})$ regarded as a global section of the co-tangent sheaf on $X \setminus D_{F,u}$. The first step is to realize that $X\setminus D_{F,u}$ is isomorphic to an affine variety on which it becomes simpler to analyze critical points. From now on, we denote by $X_{\mathrm{sm}}$ the smooth locus of $X$.

\begin{lemma}
\label{lem:affine-patch}
$X\subseteq\mathbb{P}^n$ be a closed subvariety and $u \in \mathbb{C}^{n+1}\setminus \{0\}$. Then the affine variety
\[
W_{F, u}=\{x\in C_X^\circ\mid u(x) = \deg(F) \}\subseteq\mathbb{C}^{n+1}
\]
is isomorphic to $X_{\mathrm{sm}}\setminus D_{F, u}$ via the projectivization $\phi(x)=[x]$.  Moreover, this isomorphism induces the following commutative diagram:
\begin{center}
\begin{tikzpicture}[>=angle 90]
\matrix(a)[matrix of math nodes,
row sep=2em, column sep=2em,
text height=1.5ex, text depth=0.25ex]
{W_{F, u}&&X_{\mathrm{sm}}\setminus D_{F,u}\\
&\mathbb{A}^1.&\\};
\path[->] (a-1-1) edge node[above]{$\phi$}(a-1-3);
\path[->] (a-1-1) edge node[below left]{$F$}(a-2-2);
\path[->] (a-1-3) edge node[below right]{$f_{F, u}$}(a-2-2);
\end{tikzpicture}
\end{center}
In particular, the scheme of zeros of $d\log(f_{F,u})$ in $X_{\mathrm{sm}}\setminus D_{F,u}$ is isomorphic through $\phi$ to the scheme of zeros of $d\log(F|_{W_{F,u}})$.
\end{lemma}

\begin{proof} 
The statement follows from the fact that the morphisms $\phi\colon W_{F, u}\rightarrow X_{\mathrm{sm}}\setminus D_{F, u}$ and $\psi\colon X_{\mathrm{sm}}\setminus D_{F, u}\rightarrow W_{F, u}$ such that $\psi([x])=\frac{\deg(F)x}{u(x)}$, are mutually inverse. The diagram commutes because
\[
f_{F, u}(\phi(x))=\deg(F)^{\deg(F)}\frac{F(x)}{u(x)^{\deg(F)}}=\deg(F)^{\deg(F)}\frac{F(x)}{\deg(F)^{\deg(F)}}=F(x).\qedhere
\]
\end{proof} 

The next lemma gives the ideal of the subscheme of $W_{F,u}$ of the vanishing of $d\log(F|_{W_{F,u}})$, which corresponds to the vanishing of $d\log(f_{F,u})$ in $X_{\mathrm{sm}}\setminus D_{F,u}$ by the above lemma.

\begin{lemma}
\label{how-to-really-compute-geom-MLdeg}
Let $u  \in \mathbb{C}^{n+1}\setminus \{0\} $. Then there exist a finite affine covering $W_{F,u}=\cup_{\sigma}V_{\sigma}$ such that the ideal associated to the scheme-theoretic zero locus on $V_{\sigma}$ of the differential of $\log(F|_{W_{F, u}})$ is given by
\[
I_{c+2}(B(x, u))  \subseteq \mathbb{C}[V_{\sigma}],
\]
where $B(x,u)$ is the following matrix:
\[
B(x, u)=\begin{bmatrix}
\frac{1}{F}\nabla F\\
u \\
\nabla g_1\\
\vdots\\
\nabla g_k \\
\end{bmatrix}.
\]
\end{lemma}

\begin{proof}
Define $g_0(x) = u(x) - \deg(F)$. By Lemma~\ref{lem:affine-patch} we have that $W_{F,u}$ is a smooth and dense open subset of the vanishing locus of $\langle g_0, g_1,\ldots,g_k \rangle$. Thus, we can cover $W_{F,u}$ with open patches where at least one of the $(c+1)$-minors of the Jacobian matrix of $g_0,g_1,\ldots,g_k$ does not vanish.

Let $V_{\sigma}\subseteq W_{F,u}$ be the affine open subset where the $(c+1)$-minor $\sigma$ is nonzero (we can choose finitely many of these to cover $W_{F,u}$). Let $\alpha=\{\alpha_0,\ldots,\alpha_c\}\subseteq\{0,\ldots,k\}$ and $\beta=\{\beta_0,\ldots,\beta_c\}\subseteq\{0, \ldots, n\}$ be the rows and columns of the Jacobian matrix associated to $\sigma$. Notice that $g_0(x)$ is necessary for $W_{F,u}\subseteq\mathbb{C}^{n+1}$ to attain codimension equal to $c+1$, so the index $0$ has to appear in $\alpha$. Thus, we can assume without loss of generality that $\alpha_0 = 0$.

Consider the exact sequence \cite[Chapter~II, Theorem~8.17]{Har77}
\[
0\rightarrow N_{W_{F, u}/\mathbb{C}^{n+1}}^\vee\rightarrow \Omega_{\mathbb{C}^{n+1}}^1|_{W_{F, u}}\rightarrow \Omega_{W_{F, u}}^1\rightarrow 0,
\]
and recall that $\Omega_{\mathbb{C}^{n+1}}^1|_{W_{F, u}}$ is the sheaf of $\mathcal{O}_{W_{F, u}}$-modules given by $\oplus_{i=0}^n \mathcal{O}_{W_{F, u}} \cdot dx_{i} $. Let $R = \mathbb{C}[V_{{\sigma}}]$. We recall how the above sequence simultaneously trivializes and splits on $V_{\sigma}$. Let $\gamma=\{\gamma_1<\ldots<\gamma_N\}$ be the set of column indices complimentary to $\beta$, where we set $N=n+1-c-1 = \dim(X) = \dim(W_{F, u})$. Consider the matrix
\[
K = \begin{bmatrix}
&&&& \\
\nabla g_{\alpha_0}^\intercal & \ldots & \nabla g_{\alpha_c}^\intercal & e_{\gamma_1}^\intercal & \ldots & e_{\gamma_N }^\intercal \\
&&&& \\
\end{bmatrix} \in M_{(n+1) \times (n+1)}(R),
\]
where $e_{\gamma_i}$ denotes the vector of length $n+1$ with $1$ in position $\gamma_i$ and zero in the other positions. The determinant of $K$ is $\sigma\in R$, which is a unit and thus $K$ is invertible on $V_{\sigma}$. This makes $K$, when regarded as an $R$-module map, an isomorphism of $R$-modules
\begin{equation}
\label{split-of-global-section}
 \left(\bigoplus_{i = 0}^c R \cdot dg_{\alpha_i} \right) \oplus  \left(\bigoplus_{i  \in \gamma} R\cdot dx_i \right) \cong \bigoplus_{i=0}^n R\cdot dx_i,
\end{equation}
where the first summand in parentheses is the image of $N_{W_{F, u}/\mathbb{C}^{n+1}}^\vee(V_{\sigma})\rightarrow\Omega_{\mathbb{C}^{n+1}}^1(V_{\sigma})$. Therefore,
\begin{equation}
\label{trivialco-tangent}
\Omega_{W_{F,u}}^1(V_{\sigma}) \cong \left(\bigoplus_{i=0}^n R\cdot dx_{i}\right) \Big/ \left(\bigoplus_{i = 0}^c R \cdot dg_{\alpha_i} \right) \cong \bigoplus_{i  \in \gamma} R\cdot [dx_{i}],
\end{equation}
which means that the set of equivalence classes $\{[dx_{i}]\}_{i  \in \gamma}$ generates $\Omega_{W_{F,u}}^1(V_{\sigma})$ as a free $R$-module. 
 
We have that $[d\log(F|_{W_{F,u}})] = [\sum_{i=0}^n \frac{1}{F}\frac{\partial F}{\partial x_i} dx_{i}] = \sum_{i=0 }^n \frac{1}{F}\frac{\partial F}{\partial x_i} [dx_{i}] \in \Omega_{W_{F, u}}^1(V_{{\sigma}})$. Using \eqref{trivialco-tangent}, we can argue that on $V_{\sigma}$ we can write in a unique way
\begin{align*}
\sum_{i=0}^n \frac{1}{F}\frac{\partial F}{\partial x_i} [dx_{i}] &=    \sum_{k \in \gamma } \eta_{k} [dx_{k}],
\end{align*}
for some $\eta_{k} \in R$. By definition, the $\eta_{k}$ are the regular functions generating the defining ideal of the vanishing scheme of $d\log(F|_{W_{F, u}})$ on $V_{\sigma}$. To prove the lemma it is enough to show that the ideal $J = \langle\eta_{k}\mid k\in\gamma\rangle$ is equal to $I_{c+2} (B(x, u))$ on $V_{\sigma}$.

From \eqref{split-of-global-section} we see that there is a unique way to write
\begin{align*}
\sum_{i = 0}^n \frac{1}{F}\frac{\partial F}{\partial x_i} dx_{i} &= \sum_{r=0}^c \lambda_r dg_{\alpha_r} + \sum_{i \in \gamma } \eta_{i} dx_{i}\\
&= \sum_{r=0}^c \lambda_{r} \left(\sum_{k=0}^n \frac{\partial g_{\alpha_r}}{\partial x_{k}} dx_{k}\right) + \sum_{i \in \gamma } \eta_{i} dx_{i}.
\end{align*}
This yields some relations appearing as the coefficients of the differentials $dx_j$. With our notation we can summarize the above relations in the following system of equations:
\begin{equation}
\label{augmented-system-forv(M^-1)}
K \cdot 
\begin{bmatrix}
\lambda_0 \\
\vdots \\
\lambda_c\\
\eta_{\gamma_1} \\
\vdots \\
\eta_{\gamma_{N}}
\end{bmatrix}
= 
\frac{1}{F}\nabla F^\intercal.
\end{equation}
If $1\leq i\leq N$, then by Cramer's rule we have that
\[
\eta_{\gamma_i}=\frac{\det([\nabla g_{\alpha_0}^\intercal\ldots\nabla g_{\alpha_c}^\intercal e_{\gamma_1}^\intercal\ldots \frac{1}{F}\nabla F^\intercal\ldots e_{\gamma_N}^\intercal])}{{\sigma}},
\]
where recall $\det(K)=\sigma$. Notice that the numerator is in the ideal $I_{c+2}(B(x, u))$. Thus $J \subseteq I_{c+2} (B(x, u))$. To conclude that the ideals are equal, we prove that, in the quotient $R/J$, the ideal $I_{c+2}(B(x, u))/J$ is the zero ideal. So let $\nu$ be an arbitrary minor of size $c+2$ of the matrix $B(x, u)$. The minor $\nu$ may or may not involve the row corresponding to $\frac{\nabla F}{F}$. If it does not, then $\nu=0$ because $W_{F, u}$ is smooth of codimension $c+1$. If it does, then, using the system \eqref{augmented-system-forv(M^-1)}, for all $i\in\{0,\ldots,n\}$, there exists $\ell(i)\in\{1,\ldots,N\}$ such that $\frac{1}{F}\frac{\partial F}{\partial x_i}=\sum_{r=0}^c\lambda_r\frac{\partial g_{\alpha_r}}{\partial x_{i}}+\eta_{\gamma_{\ell(i)}}$. So, in the quotient by $J$, we have that
\[
\left[\frac{1}{F}\frac{\partial F}{\partial x_i}\right]=\sum_{r=0}^c\left[\lambda_r\frac{\partial g_{\alpha_r}}{\partial x_{i}}\right]\implies\left[\frac{\nabla F}{F}\right]=\left[\sum_{r=0}^c\lambda_r\nabla g_{\alpha_r}\right].
\]
This implies that the rows of $B(x, u)$ are linear combinations of the $c+1$ vectors $\nabla g_{\alpha_r}$, hence the $c+2$ minor $\nu$ is zero in the quotient by $J$.
\end{proof}

The following is one last technical result that we need for the proof of Theorem~\ref{Gaussian-rational-schemes-are-equal}. In summary, we show that the critical points of $\ell_{F, u}$, as in \eqref{eq:log-function-of-reference}, lie inside the hyperplane $u(x)-\deg(F)=0$, which is independent from $X$.

\begin{lemma} 
\label{lem: trace}
Let $X=V(g_1,\ldots,g_k)\subseteq\mathbb{P}^n$ be a projective variety of codimension $c$. Then, over the principal open subset $U_\mu$ (see Definition~\ref{def:principal-open-from-minor-of-Jacobian-matrix}), we have
$$
u(x) - \deg(F) \in 
I_{c+1}(A(x,u))  \subseteq \mathbb{C}[U_\mu].
$$
In particular, the critical points of $\ell_{F,u}|_{C_X^\circ}$ lie on $W_{F,u}$ (see Lemma~\ref{lem:affine-patch}).
\end{lemma} 

\begin{proof}
Let $\alpha\subseteq\{1,\ldots,k\}$ be the set of rows associated to the size $c$ minor $\mu$. We show that $u(x) - \deg(F)$ is zero on the open subset 
$$
\pi_1^{-1}(U_\mu) \cong \mathrm{Spec}\left(\mathbb{C}[U_\mu][u_0, \ldots, u_n]/I_{c+1}(A(x,u))\right).
$$
By Lemma~\ref{technical-lemma-Cramer-rule} applied to the matrix $A(x,u)$ we have that
\[
\frac{\nabla F}{F} - u =\sum_{i=1}^c\lambda_i\nabla g_{\alpha_i}(x),
\]
where each entry of the vectors above is an element of the coordinate ring of $\pi_1^{-1}(U_\mu)$, because $\mu$ is a unit on $U_\mu$ and therefore $\mathbb{C}[\pi_1^{-1}(U_\mu)]_\mu = \mathbb{C}[\pi_1^{-1}(U_\mu)]$. Notice that
\begin{align*}
u(x) = &\sum_{k=0}^n u_{k} x_{k} \\
=& \sum_{k=0}^n \left( \frac{\nabla F}{F} - \sum_{i=1}^c\lambda_i\nabla g_{\alpha_i}(x)\right)_{k} x_{k} \\
=& \sum_{k=0}^n \frac{1}{F}\frac{\partial F}{\partial x_k} \cdot x_{k} - \sum_{i=1}^c\lambda_i \sum_{k=0}^n \frac{\partial g_{\alpha_i}}{\partial x_k} \cdot x_{k}
\\
=& \frac{\deg(F) F}{F} - \sum_{i=1}^c \lambda_i \mathrm{deg}(g_{\alpha_i}) \cdot  g_{\alpha_i}(x) = \deg(F) - 0 = \deg(F),
\end{align*}
where $g_{\alpha_i}(x) = 0$ because  $x \in U_\mu \subseteq C_X$ and the equalities $\sum_{k=0}^n \frac{1}{F}\frac{\partial F}{\partial x_k} \cdot x_{k}=\frac{\deg(F) F}{F}$, $\sum_{k=0}^n \frac{\partial g_{\alpha_i}}{\partial x_k} x_{k}=\mathrm{deg}(g_{\alpha_i}) \cdot g_{\alpha_i}(x)$ follow from Euler's formula. 
\end{proof}

\begin{theorem}
\label{Gaussian-rational-schemes-are-equal}
Let $F$ be a homogeneous polynomial on $\mathbb{P}^n$ and $u \in \mathbb{C}^{n+1} \setminus \{0\}$. Let $X=V(g_1,\ldots,g_k)\subseteq\mathbb{P}^n$ be an irreducible closed subvariety. Then, the projectivization map $\mathbb{C}^{n+1} \setminus \{0\} \to \mathbb{P}^n$ induces an isomorphism between the schemes of critical points of $\ell_{F,u}|_{C_X^\circ }$ and $\log(f_{F,u})$ on $X_{\mathrm{sm}}\setminus D_{F,u}$.
\end{theorem}

\begin{proof}
Fix any $u\in\mathbb{C}^{n+1}$ non-zero. By Lemma~\ref{lem:affine-patch} we have that the scheme of critical points of $\log(f_{F,u})$ on $X_{\mathrm{sm}}\setminus D_{F,u}$ is isomorphic (by projectivization) to the critical points of $\log(F|_{W_{F,u}})$ on $W_{F,u}$. Furthermore, by Lemma~\ref{how-to-really-compute-geom-MLdeg}, we have a covering of affine open subsets $V_{\sigma}\subseteq W_{F,u}$ on which the scheme of critical points of $\log(F|_{W_{F,u}})$ is given by
\[
\mathrm{Spec}(\mathbb{C}[V_\sigma ]/I_B), \qquad I_B:=I_{c+2}(B(x, u)).
\]
Recall from Lemma~\ref{how-to-really-compute-geom-MLdeg} that $\sigma$ is a $(c+1)\times(c+1)$-minor of the Jacobian matrix of $g_0,\ldots,g_k$, where $g_0=u(x)-\deg(F)$. The minor $\sigma$ involves the row that is equal to $u$ because, otherwise, there would be a non-zero $(c+1)\times(c+1)$-minor of the Jacobian of $C_X^\circ$, which contradicts the the fact that $C_X^\circ$ is smooth and of codimension $c$. Denote by $\alpha_0,\ldots,\alpha_c\in\{0,\ldots,k\}$ the indices of the rows involved in $\sigma$, with $\alpha_0=0$. Let $\beta_0,\ldots,\beta_c\in\{0,\ldots,n\}$ be the columns involved in $\sigma$. By expanding $\sigma$ with respect to the row $\alpha_0$ we obtain that
\[
\sigma=\sum_{j=0}^cu_{\beta_j}\mu_j,
\]
where the $\mu_j$ are size-$c$ minors. As these are not all simultaneously zero on $V_\sigma$, let $\mu\in\{\mu_j\}_j$ be nonzero. Over the principal open subset $U_\mu$ (see Definition~\ref{def:principal-open-from-minor-of-Jacobian-matrix}), the scheme of critical points of $\log(\ell_{F,u}|_{C_X^\circ })$ is
\[
\mathrm{Spec}(\mathbb{C}[U_\mu]/I_A), \qquad I_A:=I_{c+1}(A(x, u)).
\]
(Recall the matrix $A(x,u)$ from Definition~\ref{def:A(x,eta)-and-critical-points-ell}.)

We have then constructed the affine open subsets $U_\mu\subseteq C_X^\circ$ and $V_\sigma\subseteq W_{F,u}$ which have nonempty intersection. We observe that $U_\mu\cap V_\sigma$ is an affine subvariety of the open subset of $\mathbb{C}^{n+1}$ where $F\cdot\mu\cdot\sigma\neq0$. In particular, the ideals $I_A$ and $I_B$ can be viewed in the ring
$$
R=\mathbb{C}[U_\mu \cap V_\sigma]=\frac{\mathbb{C}[x_0,\ldots,x_n,1/F,1/\mu,1/\sigma]}{\langle g_0, \ldots, g_k\rangle}.
$$
By Lemma~\ref{lem: trace} and by the construction of $\mu,\sigma$, then we can capture all the critical points on $C_X^\circ$ and $W_{F,u}$, together with their scheme structures, as closed subschemes of $U_\mu \cap V_\sigma$. To prove that the critical points match in $U_\mu \cap V_\sigma$, it suffices to show that, within the ambient ring $R$, the ideals $I_A$ and $I_B$ are equal. This is done by showing that
\[
I_A= \langle 0 \rangle~\textrm{in}~R_B=R/I_B~\textrm{and}~I_B = \langle 0 \rangle~\textrm{in}~R_A=R/I_A.
\]

Let us first prove the former. We view the entries of the matrix $A(x,u)$ as elements of the ring $R_B$. By Lemma~\ref{technical-lemma-Cramer-rule} applied to $B(x, u)$ one can write $\frac{\nabla F}{F} = \lambda_0 u + \sum_{i=1}^{c} \lambda_i \nabla g_{\alpha_i}$ in the ring $R_B$. By taking the scalar product with the vector $(x_0, \ldots, x_n)$ and applying Euler's formula, one obtains that
\begin{align*}
    \deg(F) = \frac{\nabla F(x)}{F(x)} \cdot x &= \lambda_0 u(x) + \sum_{i=1}^{c} \lambda_i \nabla g_{\alpha_i}(x)\cdot x\\
    &= \lambda_0 u(x) + \sum_{i=1}^{c} \lambda_i \mathrm{deg}(g_{\alpha_i})g_{\alpha_i}(x)\\
    &= \lambda_0 u(x) + \sum_{i=1}^{c} \lambda_i \mathrm{deg}(g_{\alpha_i})\cdot0=\lambda_0 \deg(F) \in R_B,
\end{align*}
and thus $\lambda_0=1$. From this we argue that
\[
A(x,u)=\begin{bmatrix}
\frac{\nabla F}{F} - u \\
\nabla g_1\\
\vdots\\
\nabla g_k \\
\end{bmatrix}=
\begin{bmatrix}
\sum_{i=1}^{c} \lambda_i \nabla g_{\alpha_i} \\
\nabla g_1\\
\vdots\\
\nabla g_k \\
\end{bmatrix},
\]
where the entries are regarded as elements in $R_B$. We can conclude that $I_A = \langle 0 \rangle$ in $R_B$. 

The next step is to show that $I_B= \langle 0 \rangle$ in $R_A$, for which is enough to show that $I_{c+2}(B(x, u))$ is zero in $R_A$. We use similar ideas as above. By applying Lemma~\ref{technical-lemma-Cramer-rule} to $A(x, u)$ we have that there exist $\lambda_1,\ldots,\lambda_c\in R_A$ such that
\[
\frac{\nabla F(x)}{F(x)}-u(x)=\sum_{i=1}^c\lambda_i\nabla g_{\alpha_i}(x).
\]
Hence, again in the ring $R_A$, we can write $B(x,u)$ as follows:
\[
B(x, u)=\begin{bmatrix}
u + \sum _{i=1}^c \lambda_i \nabla g_{\alpha_i}\\
u \\
\nabla g_1\\
\vdots\\
\nabla g_k \\
\end{bmatrix}.
\]
Thus, we have $I_{c+2}(B(x,u)) = \langle 0 \rangle$ in $R_A$, which implies that $I_B = \langle 0 \rangle$ in $R_A$.
\end{proof}

The corollary below follows from Theorem~\ref{Gaussian-rational-schemes-are-equal} by taking a generic $u$.

\begin{corollary}
\label{Gaussian-geometric-schemes-are-equal}
Let $X=V(g_1,\ldots,g_k)\subseteq\mathbb{P}^n$ be an irreducible closed subvariety. Let $F$ be a homogeneous polynomial on $\mathbb{P}^n$ and $u \in \mathbb{C}^{n+1}$ generic. Then,
\[
\#\,\textrm{critical points of}~\log(f_{F,u}) ~\textrm{on}~X_{\mathrm{sm}}\setminus D_{F,u} = \mathrm{MLD}_F(X).
\]
\end{corollary}

\begin{remark}
\label{rmk:X-degenerate-MLD-independent-choice-linear-subspace}
Let $X\subseteq\mathbb{P}^n$ be degenerate in the sense that it is contained in a proper linear subspace $L$. Then one can consider the embedding $X\subseteq L$ and $\mathrm{MLD}_{F|_L}(X)$. Because the rational functions $f_{F,u}$ and $f_{F|_L,u|_L}$ on $X$ are independent from $L$, it follows that  $\mathrm{MLD}_{F|_L}(X)=\mathrm{MLD}_F(X)$.
\end{remark}

\section{Geometric computations of ML degree}
\label{sec:geom-comp-ML-deg}
In this section we prove different geometrical characterizations of the ML degree. More precisely, we characterize the ML degree via the {\it topological Euler characteristic}, {\it projective duality}, and the {\it top Chern class}. As previously, we will consider subvarieties  $X = V(g_1, \ldots, g_k)\subseteq\mathbb{P}^n$. Moreover, $u$ will be a generic linear form on $\mathbb{P}^n$.

\subsection{MLD as topological Euler characteristic}

The proof of Theorem~\ref{thm:crit-degree-as-Euler-characteristic} relies  on the following result.

\begin{theorem}[{\cite[Equation~(2)]{STV05} as phrased in \cite[Theorem~3.10]{MRW20}}]
\label{thm:result-from-STV05}
Let $Y$ be an irreducible closed subvariety in $\mathbb{C}^n$. Let
$\ell\colon\mathbb{C}^n \to \mathbb{C}$ be a general linear function, and let $H_c$ be the hyperplane in $\mathbb{C}^n$ defined by
the equation $\ell = c$ for a general $c \in \mathbb{C}$
(i.e., there is no critical point $p$ such that $\ell(p) = c$). Then the number of critical points of $\ell|_{Y_{\mathrm{sm}}}$ is
equal to 
$$
(-1)^{\dim Y}\chi(\mathrm{Eu}_Y|_{U_c}),
$$ 
where $U_c = Y \setminus H_c$ and $\mathrm{Eu}_Y$ is the local Euler obstruction function on $Y$.
\end{theorem}

The main ideas of the proof of the following theorem are similar to the ones in the proof of \cite[Theorem~1.3]{MRW21}.

\begin{theorem}
\label{thm:crit-degree-as-Euler-characteristic}
Let $F\in\mathbb{C}[x_0,\ldots,x_n]$ be a homogeneous polynomial on $\mathbb{P}^n$ and let $X=V(g_1,\ldots,g_k)\subseteq\mathbb{P}^n$ be a closed subvariety. Then
\[
\mathrm{MLD}_F(X)=(-1)^{\dim(X)}\chi(\mathrm{Eu}_{X}|_{X\setminus D_{F,u}}),
\]
where $u\in\mathbb{C}^{n+1}$ is generic. In particular, if $X$ is smooth away from $V(F)$, then
\[
\mathrm{MLD}_F(X)=(-1)^{\dim(X)}\chi_{\mathrm{top}}(X\setminus D_{F,u}).
\]
\end{theorem}

\begin{proof}
Consider
\[
\mathbb{X}_F = V(g_1, \ldots, g_k, F(x) - 1) \subseteq \mathbb{C}^{n+1}.
\]
We show that we can translate the problem of finding critical points of $f_{F,u}$ to finding critical points of the linear function $u|_{\mathbb{X}_F}$. We start by proving that $\mathbb{X}_F$ is smooth at $p$ if and only $X$ is smooth at $[p]$. For this it is enough to show that, for each $x\in\mathbb{X}_F$, the rank of the Jacobian matrix $J(g_1,\ldots,g_k,F-1)$ at $p$ equals $1+\mathrm{rk}\,J(g_1,\ldots,g_k)(p)$. If there was a point $p$ for which $\mathrm{rk}\,J(g_1,\ldots,g_k,F-1)(p) = \mathrm{rk}\,J(g_1,\ldots,g_k)(p)$, this would imply the existence of a linear combination: $\sum_{i=1}^k\lambda_i\nabla g_i(p)=\nabla F(p)$. By taking the dot product by $p$ on both sides and applying Euler's formula, this would yield  $\deg(F)=0$, which is not the case.

Consider the projectivization map $\mathbb{C}^{n+1} \setminus \{ 0 \} \to \mathbb{P}^n$. Consider the restriction $\beta\colon V(F-1)\rightarrow\mathbb{P}^n\setminus V(F)$ of the above map. We observe that $\beta$ is \'etale: the codomain is regular, the domain is Cohen--Macaulay and the fibers of $\beta$ are equidimensional (more precisely, zero-dimensional). Then $\beta$ is flat by the Miracle Flatness Theorem \cite[Theorem~23.1]{Mat89}. Let $\alpha$ be the base change of $\beta$ to $X\setminus V(F)\subseteq\mathbb{P}^n\setminus V(F)$, whose domain is $\mathbb{X}_F$. The map $\alpha\colon\mathbb{X}_F\rightarrow X\setminus V(F)$ is an \'etale cover of degree $\deg(F)$. To prove this, let $[x]\in X\setminus V(F)$. Then the preimage $\alpha^{-1}([x])$ consists of the points $\lambda x$ for $\lambda\in\mathbb{C}$ such that $F(\lambda x)=1$. But $F(\lambda x)=\lambda^{\deg(F)}F(x)=1$ if $\lambda$ is a root of $1$ of order $\deg(F)$, so we have $\deg(F)$ distinct possibilities for $\lambda$. As $\alpha$ is flat (it is a base change of $\beta$, which is flat) and unramified, $\alpha$ is \'etale of degree $\deg(F)$.

There are two important consequences of $\alpha$ being \'etale. First, by using that the local Euler obstruction is an analytic invariant (as phrased in \cite[\S\,2]{MRW21}), we have that
\[
\chi(\mathrm{Eu}_{\mathbb{X}_F}|_{\mathbb{X}_F\setminus V(u)})=\deg(F)\cdot\chi(\mathrm{Eu}_{X}|_{X\setminus D_{F,u}}).
\]
Secondly, the logarithm of the composition
$$
(f_{F, u} \circ \alpha) (x) = \frac{\deg(F)^{\deg(F)} F(x)}{u(x)^{\deg(F)}} = \frac{\deg(F)^{\deg(F)}}{u(x)^{\deg(F)}}
$$
has exactly $\deg(F)$ distinct critical points for each critical point of $\log(f_{F, u}|_{X \setminus D_{F, u}}),$ away from $V(u(x))$. Moreover, away from $V(u(x))$, the critical points of $\log(f_{F, u} \circ \alpha)$ are the same as the critical points of $u|_{\mathbb{X}_F}$ because their gradients differ by the factor $(-\deg(F)/u(x))$.

We now show that there are no critical points $p$ of $u|_{\mathbb{X}_F}$ on $V(u(x))$. Suppose there was such a point $p$. Then the matrix $B(x,u)$ from Lemma~\ref{how-to-really-compute-geom-MLdeg} satisfies the hypothesis of  Lemma~\ref{technical-lemma-Cramer-rule} giving a linear combination  $u = \lambda_0 \frac{\nabla F}{F}(p)+ \sum_i \lambda_i \nabla g_i(p)$. Applying scalar multiplication by $p$ on both sides and using the fact that $p \in \mathbb{X}_F$ gives
\[
0 = \lambda_0 \deg(F)\text{ and thus }\lambda_0 = 0.
\]
This is equivalent to $[u] \in X^\vee$, the dual variety of $X$, which is false for generic $u$ because $X^\vee$ is a proper subvariety of $(\mathbb{P}^n)^\vee$. This shows that the number of critical points of the linear form $u|_{\mathbb{X}_F}$ is equal to $\deg(F)$ times the number of critical points of $\log(f_{F, u})$ on $X \setminus D_{F, u}$.

Another consequence of the above argument is that $0$ is not a bifurcation value for $u|_{\mathbb{X}_F}$. Therefore, we can apply Theorem~\ref{thm:result-from-STV05} to argue that the number of critical points of $u|_{\mathbb{X}_F}$ is equal to $(-1)^{\dim \mathbb{X}_F}\chi(\mathrm{Eu}_{\mathbb{X}_F}|_{\mathbb{X}_F\setminus V(u(x))})$. Notice that $\alpha|_{\mathbb{X}_F \setminus V(u(x))}$ is a degree $\deg(F)$ \'etale cover of $X \setminus D_{F, u}$ because it is the base change of $\alpha$ to the open set $X\setminus D_{F, u}$. In conclusion,
\begin{align*}
    \mathrm{MLD}_F(X)& \overset{\textrm{Cor}~\ref{Gaussian-geometric-schemes-are-equal}}{=}
    \#\,\textrm{critical points of}~\log(f_{F,u}) ~\textrm{on}~X\setminus D_{F,u} \\
    &= \frac{1}{\deg(F)} \text{ \#\,critical points of } u|_{\mathbb{X}_F}\\
    &\overset{\textrm{Thm}~\ref{thm:result-from-STV05}}{=}\frac{1}{\deg(F)} (-1)^{\dim \mathbb{X}_F}\chi(\mathrm{Eu}_{\mathbb{X}_F}|_{\mathbb{X}_F\setminus V(u)})\\
    &=(-1)^{\mathrm{dim}(X)}\chi(\mathrm{Eu}_{X}|_{X\setminus D_{F,u}}).\qedhere
\end{align*}
\end{proof}

\begin{example}
\label{ex:running-example-topological-Euler-characteristic}
Consider the example in the introduction where $\mathcal{L}\subseteq\mathbb{P}(\mathbb{S}^3)$ given by the span of $\left[\begin{smallmatrix}1&0&0\\0&0&0\\0&0&0\end{smallmatrix}\right]$, $\left[\begin{smallmatrix}0&1&0\\1&0&0\\0&0&0\end{smallmatrix}\right]$, and $\left[\begin{smallmatrix}0&0&0\\0&1&0\\0&0&1\end{smallmatrix}\right]$. We consider the function $F=\det$ restricted to $\mathcal{L}$, which describes the plane curve
\[
F = \det\left(\left[\begin{smallmatrix}x&y&0\\y&z&0\\0&0&z\end{smallmatrix}\right]\right)=z(xz-y^2).
\]
Denote by $C_1$ the line $z=0$, which is tangent to the smooth conic $C_2$ given by $xz-y^2=0$. Therefore, the divisor $D_{\det,S}$, where $S\in\mathbb{S}^n$ is generic, is given by $C_1+C_2+L$, where $L$ is a generic line in the plane $\mathcal{L}$. By Theorem~\ref{thm:crit-degree-as-Euler-characteristic} and using the inclusion-exclusion principle for $\chi_{\mathrm{top}}$, we obtain that
\begin{align*}
\mathrm{MLD}_{\det}(\mathcal{L})&=(-1)^{\dim(\mathcal{L})}\chi_{\mathrm{top}}(\mathcal{L}\setminus D_{\det,S})\\
&=\chi_{\mathrm{top}}(\mathbb{P}^2)-\chi_{\mathrm{top}}(C_1\cup C_2\cup L)\\
&=3-(\chi_{\mathrm{top}}(C_1)+\chi_{\mathrm{top}}(C_2)+\chi_{\mathrm{top}}(L)-\chi_{\mathrm{top}}(4~\textrm{points}))\\
&=3-(2+2+2-4)=1.
\end{align*}
\end{example}

\begin{example}[Formula for the ML degree of a curve]
\label{ex: MLDofcurve}
Let $F$ be a homogeneous polynomial on $\mathbb{P}^n$. Let $C\subseteq\mathbb{P}^n$ be a degree $d$ curve smooth away from $V(F)$. Then the Gaussian ML degree of $C$ can be computed using Theorem~\ref{thm:crit-degree-as-Euler-characteristic}. Let $C\cap V(F)=\{p_1,\ldots,p_k\}$. Let $\nu\colon C'\rightarrow C$ be the normalization of $C$. For each $i\in\{1,\ldots,k\}$, let $h_i$ be the cardinality of the fiber $\nu^{-1}(p_i)$. Let $g$ be the geometric genus of $C'$. Then,
\[
\mathrm{MLD}_F(C)=(-1)^{\dim C}\chi_{\mathrm{top}}(C\setminus D_{F,u})=-2+2g+d+\sum_{i=1}^kh_i.
\]
It will be interesting to generalize to higher dimensional varieties the above result for curves.
\end{example}

\begin{remark}
Note from Theorem~\ref{thm:crit-degree-as-Euler-characteristic} that $\mathrm{MLD}_F(X)$ equals $\mathrm{MLD}_G(X)$, where $G=\mathrm{rad}(F)$. This is because $D_{F,u}=D_{G,u}$.
\end{remark}

\subsection{MLD via projective duality}

Let $(\mathbb{P}^n)^{\vee}$ denote the dual projective space. Recall that the dual variety of a projective variety $Y\subseteq\mathbb{P}^n$ is the irreducible variety $Y^{\vee}\subseteq(\mathbb{P}^n)^{\vee}$ defined by the Zariski closure of
\[
\{H\in(\mathbb{P}^n)^{\vee}\mid H\text{ is tangent to }Y\text{ at some smooth point }y\in Y_{\mathrm{sm}},\text{ i.e. }T_y(Y)\subseteq H\},
\]
where $T_y(Y)$ denotes the projective tangent space to $Y$ at $y$.

\begin{theorem}
\label{thm:MLD-via-dual-variety}
Let $F\in\mathbb{C}[x_0,\ldots,x_n]$ be a homogeneous polynomial and assume that $X=V(g_1,\ldots,g_k)\subseteq\mathbb{P}^n$ is an irreducible variety. Let $\mathbb{X}_F\subseteq\mathbb{C}^{n+1}$ as in the proof of Theorem~\ref{thm:crit-degree-as-Euler-characteristic} and denote by $\overline{\mathbb{X}}_F$ its closure in $\mathbb{P}^{n+1}$. Let $m(\overline{\mathbb{X}}_F^\vee)$ be the multiplicity of $\overline{\mathbb{X}}_F^\vee$ at the point $[0:\ldots:0:1]$. Then
\begin{displaymath}
\mathrm{MLD}_F(X) = \left\{ \begin{array}{ll}
\frac{\deg(\overline{\mathbb{X}}_F^{\vee})-m(\overline{\mathbb{X}}_F^\vee)}{\deg(F)} &\textrm{if}~\overline{\mathbb{X}}_F^{\vee}~\text{is a hypersurface},\\
0&\textrm{otherwise}.
\end{array} \right.
\end{displaymath}
\end{theorem}

\begin{proof}
The proof follows the ideas in \cite[Theorem~13]{C+21}. As in the proof of Theorem~\ref{thm:crit-degree-as-Euler-characteristic}, we have that $\mathrm{MLD}_F(X)$ equals the number of critical points of  a generic linear function $u|_{\mathbb{X}_F}$, divided by $\deg(F).$  We observe that $p\in\mathbb{X}_F$ is a critical point for $u|_{\mathbb{X}_F}$ if and only if $p$ is smooth and the tangent space to $\mathbb{X}_F$ at $p$ is contained in a hyperplane of the form $u(x)-c=0$ for some constant $c$. Therefore, we will prove the following: (1) If $\overline{\mathbb{X}}_F^\vee$ is not a hypersurface, then $\mathrm{MLD}_F(X)=0$, so we assume $\overline{\mathbb{X}}_F^\vee$ is a hypersurface; (2) For a generic choice of $u$, the tangent hyperplanes to $\mathbb{X}_F$ of the form $u(x)-c=0$ are tangent at exactly one point; (3) The number of such hyperplanes is finite and equal to  $\deg(\overline{\mathbb{X}}_F^\vee)-m(\overline{\mathbb{X}}_F^\vee)$. Putting these facts together yields the formula for $\mathrm{MLD}_F(X)$ in the statement.

To prove (1), (2), and (3) we need some preliminary considerations. Let $y=[x_0:\ldots:x_n:t]$ be the coordinate of $\mathbb{P}^{n+1}$, so that the closure in $\mathbb{P}^{n+1}$ of a hyperplane $u(x)-c=0$ is given by $u(x)-ct=0$. For fixed $u$, consider the line $L_u\subseteq(\mathbb{P}^{n+1})^\vee$ given by the pencil of all the hyperplanes of the form
\[
 \lambda u(x)+\mu t=0,
\]
where $[\lambda:\mu]\in\mathbb{P}^1$. We observe that, regardless of the choice of $u$, the line $L_u$ passes through the point $[0:\ldots:0:1]\in(\mathbb{P}^{n+1})^\vee$, i.e. the hyperplane at infinity $t=0$ (obtained by setting $\lambda = 0$) is part of $L_u$. We now show that there is an open subset $\mathcal{U} \subseteq \overline{\mathbb{X}}_F^\vee$ such that the points in the intersection $L_u \cap \mathcal{U}$ correspond to the tangent hyperplanes to $\mathbb{X}_F$ of the form $u(x)-c=0$. For this purpose, let us first consider the conormal variety
\[
C(\overline{\mathbb{X}}_F) := \overline{ \{ (y, H) \in \mathbb{P}^{n+1} \times (\mathbb{P}^{n+1})^\vee\mid y \in (\overline{\mathbb{X}}_F)_{\mathrm{sm}}~\textrm{and}~T_y(\overline{\mathbb{X}}_F) \subseteq H \} }
\]
with projection maps $\pi_1\colon C(\overline{\mathbb{X}}_F) \twoheadrightarrow \overline{\mathbb{X}}_F$ and $ \pi_2\colon C(\overline{\mathbb{X}}_F) \twoheadrightarrow \overline{\mathbb{X}}_F^\vee $. We then define
\begin{align*}
U&:=\pi_1^{-1}((\mathbb{X}_F)_{\mathrm{sm}})\cap\pi_2^{-1}((\overline{\mathbb{X}}_F^\vee)_{\mathrm{sm}}),\\
B&:=\pi_2(C(\overline{\mathbb{X}}_F)\setminus U)\setminus\{[0:\ldots:0:1]\},\\
\mathcal{U}&:=\overline{\mathbb{X}}_F^\vee\setminus(B\sqcup\{[0:\ldots:0:1]\}).
\end{align*}
We emphasize that $U\neq\emptyset$ because it is the intersection of two nonempty open subsets of $C(\overline{\mathbb{X}}_F)$, which is irreducible of dimension $n$. By definition, the points in $\mathcal{U}$ correspond to the tangent hyperplanes to $\mathbb{X}_F$ at smooth points. Moreover, biduality (see \cite[Theorem~1.1]{GKZ94}) gives that the contact locus for every $H \in \mathcal{U}$ is a linear space of dimension equal to $\mathrm{codim}_{(\mathbb{P}^{n+1})^\vee}(\overline{\mathbb{X}}_F^\vee) - 1$.

Let us prove (1). From our set-up, the set of tangent hyperplanes to $\mathbb{X}_F$ of the form $u(x)-c=0$ is a subset of $L_u \cap \overline{\mathbb{X}}_F^\vee \setminus \{[0:\ldots:0:1]\}$. We show that when $\overline{\mathbb{X}}_F^\vee$ is not a hypersurface this intersection is empty for generic $u$, and thus $\mathrm{MLD}_F(X)=0$. By contradiction, assume that $L_u \cap \overline{\mathbb{X}}_F^\vee \setminus \{[0:\ldots:0:1]\} \neq \emptyset$ for generic $u$. This is equivalent to the fact that the map $\overline{\mathbb{X}}_F^\vee \dashrightarrow(\mathbb{P}^{n})^\vee$, obtained by restricting the projection from $[0:\ldots:0:1]$, is dominant. This is impossible because the dimension of $\overline{\mathbb{X}}_F^\vee$ was assumed to be strictly less than $(n+1)-1=n$. This concludes the proof of (1).

From now on we assume that $\overline{\mathbb{X}}_F^\vee$ is a hypersurface. To show (2) it is enough to prove that, for a generic choice of $u$,
\[
(L_u\cap\overline{\mathbb{X}}_F^\vee)\setminus\{[0:\ldots:0:1]\}\subseteq\mathcal{U}.
\]
This suffices because then the biduality theorem guarantees that the contact locus of each hyperplane $u(x)-c=0$ tangent to $\mathbb{X}_F$ is a linear space of dimension $1-1=0$. By contradiction, assume that for a generic $u$ there exists a point in $L_u\cap B$. We then argue in a way analogous to (1) above, as the restriction to $B$ of the projection $(\mathbb{P}^{n+1})^\vee\dashrightarrow(\mathbb{P}^{n})^\vee$ from $[0:\ldots:0:1]$ is dominant. This is a contradiction because the dimension of $B$ is strictly less than $n$, because
\[
\dim \overline{B} = \dim\pi_2(C(\overline{\mathbb{X}}_F^\vee)\setminus U)\leq\dim(C(\overline{\mathbb{X}}_F^\vee)\setminus U)\leq\dim C(\overline{\mathbb{X}}_F^\vee)-1=n-1.
\]
This finishes the proof of (2).

Finally, for (3) we compute the number of points in $(L_u\cap\overline{\mathbb{X}}_F^\vee)\setminus\{[0:\ldots:0:1]\}$. Because of the genericity of $u$, the intersection $L_u\cap\overline{\mathbb{X}}_F^\vee$ is $\deg(\overline{\mathbb{X}}_F^\vee)$, and the points in this intersection have multiplicity $1$, except for $[0:\ldots:0:1]$ which may have a different multiplicity. As $L_u$ is generic through $[0:\ldots:0:1]$, we have that this point contributes to the count $\deg(\overline{\mathbb{X}}_F^\vee)$ by $m(\overline{\mathbb{X}}_F^\vee)$, which we then have to remove from the count as it does not correspond to critical points of $u|_{\mathbb{X}_F}$. This gives the claimed formula for $\mathrm{MLD}_F(X)$.
\end{proof}

\begin{example}
\label{ex:running-example-dual-variety}
Let $\mathcal{L}\subseteq\mathbb{P}(\mathbb{S}^3)$ and $F=\det|_{\mathcal{L}}$ as in Example~\ref{ex:running-example-topological-Euler-characteristic}. Since $\mathcal{L}$ is already a projective linear space, by Remark~\ref{rmk:X-degenerate-MLD-independent-choice-linear-subspace} it is unnecessary to consider the embedding into $\mathbb{P}(\mathbb{S}^3)$ and have $\mathcal{L}$ itself as our ambient space. So, $\mathbb{X}_{\det} = V(z(xz - y^2) - 1) \subseteq C_\mathcal{L} \cong \mathbb{C}^3$  and, using Macaulay2 \cite{M2,Sta}, we can compute the dual variety of $\overline{\mathbb{X}}_{\det}=V(xz^2-y^2z-t^3)$ to be
\[
\overline{\mathbb{X}}_{\det}^\vee = V(27y^4-216xy^2z+432x^2z^2+64xt^3),
\]
where we identified $\mathbb{P}^3 \cong (\mathbb{P}^3)^\vee$. In this example $[0:\ldots:0:1] \in \overline{\mathbb{X}}_F^\vee$ is a smooth point. So, by Theorem~\ref{thm:MLD-via-dual-variety}, we obtain that:
\[
\mathrm{MLD}_F(\mathcal{L})=\frac{\deg(\overline{\mathbb{X}}_F^\vee)-1}{\deg(F)}=1,
\]
which agrees with the computation in Example~\ref{ex:running-example-topological-Euler-characteristic}.
\end{example}

\begin{example}
Suppose $X \cong \mathbb{P}^2$ and suppose $F = x^3 - y^2z$ is a cuspidal curve. Then $\mathbb{X}_F = V(x^3 - y^2z - 1) \subseteq \mathbb{C}^3$ and the dual variety of $\overline{\mathbb{X}}_F$, again computed using Macaulay2, is given by
\[
\overline{\mathbb{X}}_F^\vee = V(16x^6+216x^3y^2z+729y^4z^2+32x^3t^3-216y^2zt^3+16t^6),
\]
where we identified $\mathbb{P}^3 \cong (\mathbb{P}^3)^\vee$. Since $[0:\ldots:0:1]\not\in \overline{\mathbb{X}}_F^\vee$, by Theorem~\ref{thm:MLD-via-dual-variety} we obtain that
\[
\mathrm{MLD}_F(X)=\frac{1}{\deg(F)}\deg(\overline{\mathbb{X}}_F^\vee)=\frac{6}{3}=2. 
\]
\end{example}

\begin{example}
Consider $\mathcal{L}=\mathbb{P}^{n}$ and define
\[
F(x) = \left(\sum_{i=0}^nc_ix_i\right)^d, \quad c_i \in \mathbb{C},~d\in\mathbb{Z}_{>0}.
\]
Then $\mathbb{X}_{F}=V((\sum_i c_ix_i)^d-1)\subseteq\mathbb{C}^{n+1}$, hence $\overline{\mathbb{X}}_F\subseteq\mathbb{P}^{n+1}$ is a disjoint union of $d$ hyperplanes. Notice that the dual variety of a linear subspace is never a hypersurface if $n\geq2$. It follows that $\mathrm{MLD}_F(\mathcal{L})=0$ if $n\geq2$.
\end{example}

\subsection{MLD via Chern classes}

Let $X\subseteq\mathbb{P}^n$ be an irreducible projective subvariety which is smooth away from $V(F)$. Let $D_{F, u}$ be the support of $\mathrm{div}(f_{F,u})$ and consider $D_{F, u}=C+L$, where $C$ is the vanishing locus of $F$ and $L$ the vanishing locus of a generic linear form $u$. Recall that in this setting we may consider a log resolution of singularities \cite[Definition~1.12]{Kol13} which we will denote by $\pi \colon (Y,B)\rightarrow (X,D_{F, u})$, where $B=\pi^{-1}(D_{F,u})$.

\begin{theorem}
\label{thm:Gauss-ML-degree-using-tot-Chern-class}
Let $X$, $D_{F,u}$, and $\pi$ as above and assume that
$B = \mathrm{Supp}(\mathrm{div}(f_{F, u} \circ \pi)) = \sum_{i=1}^r B_i$, where the latter equality is the decomposition into irreducible components. Then $\mathrm{MLD}_F(X)$ is the degree of the coefficient in front of $z^d$ in
\[
c_{\mathrm{tot}}(\Omega_Y^1) \prod_{i=1}^r (1 - zB_i)^{-1}\in A^*(Y)[z].
\]
\end{theorem}

\begin{proof}
By definition, $\mathrm{MLD}_F(X)$ equals the number of critical points of $f_{F,u}$ on $X\setminus D_{F,u}$. Let $\widetilde{f}_{F,u}:=f_{F, u} \circ \pi$. As we have an isomorphism $Y\setminus B\cong X\setminus D_{F,u}$ (here we used that $X$ is smooth away from $V(F)$), we have that $\mathrm{MLD}_F(X)$ is equal to the number of critical points of $\widetilde{f}_{F,u}$ on $Y\setminus B$. Note that $B$ intersects every curve on $Y$ because $D_{F,u}$ intersects every curve on $X$ as it is ample (it is linearly equivalent to a positive multiple of a hyperplane section) and $B$ intersects every exceptional divisor of $Y\rightarrow X$ because it contains them by definition. Additionally, the sheaf $\Omega_Y^1(\log B)$ is locally free (see for instance \cite[Example~2.1]{Dol07}). Hence, we can apply \cite[Corollary~5]{CHKS06} to $Y$ and the rational function $\widetilde{f}_{F,u}$ to argue that the number of critical points of $\widetilde{f}_{F,u}$ on $Y\setminus B$ equals the coefficient of $z^d$ in the above statement.
\end{proof}

\begin{example}
\label{ex:running-example-tot-CHern-class}
Let $\mathcal{L}\subseteq\mathbb{P}(\mathbb{S}^3)$ and $F=\det$ as in Example~\ref{ex:running-example-topological-Euler-characteristic}. Let us compute $\mathrm{MLD}_{\det}(X)$ using Theorem~\ref{thm:Gauss-ML-degree-using-tot-Chern-class}. Let $D_{\det,S}$ be the support of the divisor associated to the rational function $\det(M)/\mathrm{tr}(SM)^3$. A log resolution $(Y,B)$ of $(\mathcal{L},D_{\det,S})$ can be obtained as the following iterated blow up of $\mathcal{L}$:
\begin{itemize}

\item Let $\mathcal{L}_1$ be the simple blow up of $\mathcal{L}$ at the intersection point $C_1\cap C_2$. Denote by $E_1$ the exceptional divisor;

\item Let $Y$ be the simple blow up of $\mathcal{L}_1$ at the intersection point $\widehat{C}_1\cap\widehat{C}_2\cap E_1$ with exceptional divisor $E_2$.

\end{itemize}
Denote by $\widehat{C}_1,\widehat{C}_2\subseteq Y$ the strict transforms of $C_1,C_2$. For simplicity of notation, we still denote by $E_1,L\subseteq Y$ the respective strict transforms. Note that $K_Y=-3L+E_1+2E_2$. The needed intersection numbers among $\widehat{C}_1,\widehat{C}_2,L,E_1,E_2,K_Y$ are summarized in the following table:
\begin{displaymath}
\renewcommand{\arraystretch}{1.4}
\begin{array}{|c|c|c|c|c|c|c|}
\hline
\cdot&\widehat{C}_1&\widehat{C}_2&L&E_1&E_2&K_Y\\
\hline
\widehat{C}_1&-1&0&1&0&1&-1\\
\hline
\widehat{C}_2&&2&2&0&1&-4\\
\hline
L&&&1&0&0&-3\\
\hline
E_1&&&&-2&1&0\\
\hline
E_2&&&&&-1&-1\\
\hline
\end{array}
\end{displaymath}
On the other hand, $c_{\mathrm{tot}}(\Omega_Y^1)=1+K_Yz+\chi_{\mathrm{top}}(Y)z^2=1+K_Yz+5z^2$. So we need to find the coefficient of $z^2$ after expanding the following product:
\[
(1+K_Yz+5z^2)(1+\widehat{C}_1z-z^2)(1+\widehat{C}_2z+2z^2)(1+Lz+z^2)(1+E_1z-2z^2)(1+E_2z-z^2),
\]
which yields $\mathrm{MLD}_{\det}(\mathcal{L})=1$ after an explicit calculation.
\end{example}

\section{Application to Gaussian ML degree}
\label{sec:Application-to-Gaussian-ML-degree}
Consider the complex vector space $\mathbb{S}^n$ of symmetric $n\times n$ matrices.
A motivating example for studying critical points of log functions is the Gaussian maximum likelihood degree (Gaussian ML degree) of a subvariety  $X\subseteq\mathbb{P}(\mathbb{S}^n)$, defined as the number of critical points of the \emph{Gaussian log likelihood function}
\[
\ell_{S}(M)=\log(\det(M))-\mathrm{tr}(SM)
\]
restricted to $C^\circ_X\subseteq\mathbb{S}^n$ over $X$, for a choice of generic matrix $S\in\mathbb{P}(\mathbb{S}^n)$. Notice that all the elements of $\mathbb{P}(\mathbb{S}^{n})^\vee$ can be expressed in the form $u(M)=\mathrm{tr}(SM)$. Using the notation introduced in this paper, we can adopt the following equivalent definition.

\begin{definition}
For $X \subseteq \mathbb{P}(\mathbb{S}^{n})$ we define
\[
\mathrm{Gaussian\,ML\,degree}(X) = \mathrm{MLD}_{\det}(X).
\]
\end{definition}

In particular, it follows that the Gaussian ML degree can be computed as a special case of Theorem~\ref{thm:crit-degree-as-Euler-characteristic} with $F=\det$. We make this explicit in the next corollary.

\begin{corollary}
\label{cor:GaussMLasEulerchar}
Let $X\subseteq\mathbb{P}(\mathbb{S}^n)$ be a closed subvariety. Then
\[
\mathrm{Gaussian\,ML\,degree}(X)=(-1)^{\dim(X)}\chi(\mathrm{Eu}_X|_{X\setminus D}),
\]
where $D=\mathrm{Supp}\left(\mathrm{div}\left(\frac{\det(M)}{\mathrm{tr}(SM)^n}\right)\right)$ for some generic $S \in \mathbb{S}^n$.
\end{corollary}

\begin{remark}
\label{rmk:Gauss-MLD-via-polar-Cremona-transform-and-Milnor-number}
The Gaussian ML degree was first studied in \cite{SU10}, and since then it has been intensively investigated, especially in the case where $X$ is a linear space $\mathcal{L}\subseteq\mathbb{P}(\mathbb{S}^n)$. In this case, the result of Corollary~\ref{cor:GaussMLasEulerchar} was already embedded in the literature. If $F=\det(M)|_{\mathcal{L}}$, then the Gaussian ML degree of $\mathcal{L}$ equals the degree of the \textit{polar Cremona transformation} $\nabla F\colon\mathcal{L}\dashrightarrow\mathcal{L}^\vee$ by \cite[Proposition~2.4.1]{DMS21} (see \cite{Dol00,DP03} for more background about the polar Cremona transformation). Then, the degree of $\nabla F$ can be computed as a topological Euler characteristic by \cite[Theorem~1]{DP03}, yielding the formula as in Corollary~\ref{cor:GaussMLasEulerchar} for the particular case of a linear subspace. Explicitly,
\begin{equation*}
\mathrm{Gaussian\,ML\,degree}(\mathcal{L})=(-1)^{\dim\mathcal{L}} \chi_{\mathrm{top}}(\mathcal{L}\setminus(V(F)\cup H)),
\end{equation*}
where  $H$ is a generic hyperplane in $\mathcal{L}$.

There is a second geometric interpretation of the degree of $\nabla F$. Assume that $V(F)\subseteq\mathcal{L}$ has only isolated singularities. Then, by \cite[Section~3]{DP03} (for further details see \cite[Chapter~5, Section~4]{Dim92}), this can be interpreted in terms of the Milnor numbers $\mu(p)$ of $p\in V(F)\setminus V(F)_{\mathrm{sm}}$ (see for instance \cite[Page~10]{Dim92} for the definition in terms of the Jacobian ideal of $F$), which yields
\[
\mathrm{Gaussian\,ML\,degree}(\mathcal{L})=(d-1)^{\dim(\mathcal{L})}-\sum_{p\in V(F)\setminus V(F)_{\mathrm{sm}}} \mu(p).
\]
Related to this, see also \cite{Huh14}.
\end{remark}

\begin{example}
\label{ex:running-example-Milnor}
A net is a two-dimensional projective linear space $\mathcal{L}$, spanned by three non-aligned points in $\mathbb{P}(\mathbb{S}^n)$. If we define $F:=\det|_\mathcal{L}$, then the variety cut out by $F=0$ is a plane curve. Computing the Gaussian ML degree of $\mathcal{L}$ via Corollary~\ref{cor:GaussMLasEulerchar}, we can disregard the multiplicities of the irreducible components of such curve, so we can assume that $F$ is a reduced polynomial. In particular, the singularities of this curve consist of isolated points. Then, we can use the Milnor number formula in Remark~\ref{rmk:Gauss-MLD-via-polar-Cremona-transform-and-Milnor-number} to compute the Gaussian ML degree of $\mathcal{L}$. The following code in Macaulay2 for $n=3$ and $\mathcal{L}=\mathrm{Span}\left\{\left[\begin{smallmatrix}1&0&0\\0&0&0\\0&0&0\end{smallmatrix}\right], \left[\begin{smallmatrix}0&1&0\\1&0&0\\0&0&0\end{smallmatrix}\right], \left[\begin{smallmatrix}0&0&0\\0&1&0\\0&0&1\end{smallmatrix}\right]\right\}$ computes the Gaussian ML degree of $\mathcal{L}$ using the Milnor numbers of the singularities, yielding $1$. The code can be adjusted to compute any example for arbitrary $n$ and $\mathcal{L}$. The Milnor numbers are computed using the Macaulay2 code in \cite{Alu03}.

\begin{verbatim}
load "CSM.m2"
R = QQ[x,y,z]
-- Provide a basis for a net L of matrices of arbitrary size
L = { sub(matrix({{ 1, 0, 0 }, { 0, 0, 0 }, { 0, 0, 0 }}),R),
      sub(matrix({{ 0, 1, 0 }, { 1, 0, 0 }, { 0, 0, 0 }}),R),
      sub(matrix({{ 0, 0, 0 }, { 0, 1, 0 }, { 0, 0, 1 }}),R),
       }
-- Compute F
F = det(x*L_0 + y*L_1 + z*L_2)
I = radical ideal F
d = degree I
--Compute sum of Milnor numbers
milnor(I)
\end{verbatim}
Then one can compute $\mathrm{Gaussian\,ML\,degree}(\mathcal{L})$ as $(d-1)^2$ minus the sum of the Milnor numbers produced above. In this example $d=3$ and the sum of the Milnor numbers is $3$, resulting in $\mathrm{Gaussian\,ML\,degree}(\mathcal{L}) = 2^2-3=1$.
\end{example}

In \cite{AGKMS21} the authors gave a characterization of the Gaussian ML degree for linear subspaces of $\mathbb{P}(\mathbb{S}^3)$ using the inverse linear space $\mathcal{L}^{-1}$, which is defined to be the Zariski closure in $\mathbb{P}(\mathbb{S}^n)$ of the set of inverses of invertible matrices in $\mathcal{L}$ (see also \cite{MMMSV23,MMW21}). More explicitly, for a generic linear space $\mathcal{L}$, the Gaussian ML degree of $\mathcal{L}$ equals the degree of $\mathcal{L}^{-1}$. The Gaussian ML degree for $\mathcal{L}$ and $\mathcal{L}^{-1}$ in the case where $\dim(\mathcal{L})=1$ was studied in \cite{FMS21}.

\begin{remark}
Consider a variety $X \subseteq \mathbb{P}(\mathbb{S}^n)$ and let
\[
X^{-1} := \overline{\{M \in \mathbb{P}(\mathbb{S}^n)\setminus V(\det)\mid M^{-1} \in X\} }.
\]
The open subsets of invertible matrices in $X$ and $X^{-1}$ are isomorphic via the inversion map. This implies that
\[
X\setminus (V(\det) \cup H) \cong X^{-1}\setminus (V(\det) \cup H^{-1}),
\]
which can be relevant to compute Gaussian ML degrees, as the next examples show.
\end{remark}

\begin{example}
Let $X\subseteq\mathbb{P}(\mathbb{S}^3)$ and suppose that 
\[
X^{-1} = \mathrm{Span}\left( \left[\begin{smallmatrix}0&0&1\\0&1&0\\1&0&0\end{smallmatrix}\right], \left[\begin{smallmatrix}0&0&0\\0&0&1\\0&1&0\end{smallmatrix}\right], \left[\begin{smallmatrix}1&0&0\\0&0&0\\0&0&0\end{smallmatrix}\right]\right) \cong \mathbb{P}^2.
\]
Then one can reconstruct the original $X$ by considering the adjugate of a matrix of $\left[\begin{smallmatrix}z&0&x\\0&x&y\\x&y&0\end{smallmatrix}\right]$, which is
\[
\begin{bmatrix}
-y^2 & xy & -x^2 \\
xy & -x^2 & -yz \\
-x^2 & -yz & xz
\end{bmatrix},
\]
and finding the relations among its entries. Explicitly, we obtain that
\[
X = V(x_{22}-x_{13}, x_{13}x_{23}-x_{12}x_{33}, x_{12}x_{23}-x_{11}x_{33}, x_{12}^2-x_{11}x_{13}).
\]
For a generic hyperplane $H=V(\sum_{1 \leq i,j \leq 3} s_{ij}x_{ij})\subseteq\mathbb{P}(\mathbb{S}^3)$, let $Y := H^{-1} \cap X^{-1}$. In the coordinates $x,y,z$ on the plane $X^{-1}$ we can express $Y$ as 
\[
Y=V\left( (2s_{13}+s_{22})x^2 + s_{11}y^2 - 2s_{12}xy +   2s_{23}yz - s_{33}xz \right) \subseteq \mathbb{P}^2.
\]
This results in a smooth conic that passes through $[0:0:1]$. Note that $\det\left( \left[\begin{smallmatrix}z&0&x\\0&x&y\\x&y&0\end{smallmatrix}\right]\right) = -(x^3 + y^2z)=0$ is a curve with a cusp singularity at $[0:0:1]$, and the intersection multiplicity of $Y$ and $V(\det)$ at this point is $2$. One can also check that the multiplicities of the remaining intersection points between $Y$ and $V(\det)$ are equal to one. As the intersection product $Y\cdot V(\det)$ equals $6$, we must have that $Y$ and $V(\det)$ intersect in five distinct points. From this we can argue that
\begin{align*}
    \mathrm{Gaussian\,ML\,degree}(X)&= \chi_{\mathrm{top}}(X\setminus (V(\det) \cup H)) \\
    &= \chi_{\mathrm{top}}(X^{-1}\setminus (V(\det) \cup H^{-1}))\\
    &=\chi_{\mathrm{top}}(\mathbb{P}^2 \setminus (V(x^3 + y^2z) \cup Y))\\
    &=\chi_{\mathrm{top}}(\mathbb{P}^2 ) - \chi_{\mathrm{top}}(V(x^3 + y^2z)) - \chi_{\mathrm{top}}(Y) + \chi_{\mathrm{top}}(Y \cap V(F)) \\
    &= 3 - 2 - 2 + 5 = 4.
\end{align*}
\end{example}

\begin{example}
Consider a line $\mathcal{L} \subseteq \mathbb{P}(\mathbb{S}^n)$ and its inverse linear space $\mathcal{L}^{-1}$. Then the set of invertible matrices of $X$ and $\mathcal{L}$ are isomorphic via the inversion map. Note that
\begin{align*}
    \mathrm{Gaussian\,ML\,degree}(\mathcal{L}^{-1}) &=(-1)^{\dim\mathcal{L}^{-1}} \chi_{\mathrm{top}}(\mathcal{L}^{-1} \setminus (V(\det) \cup H))\\ 
    &=-\chi_{\mathrm{top}}(\mathcal{L}^{-1} \setminus V(\det))  + \deg(\mathcal{L}^{-1}),
\end{align*}
because $H$ intersects $\mathcal{L}^{-1}$ in $\deg(\mathcal{L}^{-1}) $ many points by definition of degree. Moreover, 
\begin{align*}
    -\chi_{\mathrm{top}}(\mathcal{L}^{-1} \setminus V(\det)) = -\chi_{\mathrm{top}}(\mathcal{L} \setminus V(\det)) =
    \mathrm{Gaussian~ML~degree}(\mathcal{L})  - 1.
\end{align*}
In conclusion,
\[
\mathrm{Gaussian\,ML\,degree}(\mathcal{L}^{-1}) =  \mathrm{Gaussian\,ML\,degree}(\mathcal{L}) + \deg(\mathcal{L}^{-1})  - 1.
\]
Note that this is in accordance with \cite[Theorem~4.2]{FMS21}.
\end{example}

\section{Comparison to discrete ML degree}\label{discrete}

The discrete likelihood function is a construction originating from statistics and introduced to the unfamiliar reader in the algebraic statistics setting in \cite{CHKS06,HS14,Huh13}. An equivalent starting point is to directly define the \emph{discrete likelihood function} as
\begin{align}
\begin{split}
\label{eq:disc-likelihood-function}
    f_u\colon\mathbb{P}^n&\dashrightarrow \mathbb{C}
    \\
    [x]&\mapsto \frac{\prod_{i=0}^{n} x_i^{u_i}}{(\sum_{i=0}^n x_i)^{\sum_{i}u_i}},
\end{split}
\end{align}
where $u=(u_0,\ldots,u_{n})\in (\mathbb{Z}_{\geq 0})^{n+1}$. We call $\log(f_u)$ the \textit{discrete log likelihood function} and we use the notation $D = \mathrm{Supp}(\mathrm{div} (f_u))$. The aim of \cite{CHKS06,HS14,Huh13} is to study the number of complex critical points of $\log f_u(x)|_{X \setminus D}$ for generic $u$, i.e. $u$ not lying on the solution set of certain discrete polynomial equations. The number of critical points is by definition  the \emph{discrete ML degree of $X \subseteq \mathbb{P}^n$}. Note that $D$ is independent of the choice of $u$ and $D=D_{F,(1,\ldots,1)}$ with $F=\prod_{i=0}^nx_i$.

In \cite{Huh13} it is shown that, if $X$ is smooth away from $D$, then
\[
\mathrm{Discrete~ML~degree}(X) = (-1)^{\dim(X)} \chi_{\mathrm{top}}(X \setminus D).
\]
Let $\mathbb{P}^{n} \subseteq \mathbb{P}(\mathbb{S}^{n+1})$ be the projective subspace parametrizing diagonal matrices and consider a closed subvariety $X \subseteq \mathbb{P}^{n}$. Recall from Theorem~\ref{thm:crit-degree-as-Euler-characteristic} that, assuming smoothness of $X$ away from $V(F)$,
\[
\mathrm{Gaussian~ML~degree}(X) = (-1)^{\dim(X)} \chi_{\mathrm{top}}(X \setminus D_{\det,\eta}),
\]
where $D_{\det,\eta} = \mathrm{Supp}\left(\mathrm{div} \left(\frac{\prod_{i=0}^{n} x_i}{(\sum_{i=0}^n \eta_ix_i)^{n+1}}\right)\right)$ and $\eta\in\mathbb{C}^{n+1}$ is generic. It is therefore natural to explore a connection between the two notions of ML degree. The following result is a first attempt in this direction.

\begin{proposition}
\label{ineqDiscrete}
Let $X \subseteq \mathbb{P}^n \subseteq \mathbb{P}(\mathbb{S}^{n+1})$ be a closed subvariety, where we view $\mathbb{P}^n$ as the projective subspace of diagonal matrices. Then
\[
\mathrm{Discrete~ML~degree}(X) \leq \mathrm{Gaussian~ML~degree}(X).
\]
\end{proposition}

\begin{proof}
Let $u=(u_0,\ldots,u_n)\in(\mathbb{Z}_{\geq0})^{n+1}$ be generic. We have that $\mathrm{Discrete~ML~degree}(X)$ equals the number of critical points of $\log f_u(x)|_{X \setminus \mathrm{Supp}(\mathrm{div} (f_u))}$, which are finitely many and reduced (see \cite[Theorem~1.6]{HS14}). Let $\varphi\colon\mathbb{P}^n \hookrightarrow \mathbb{P}(\mathbb{S}^{\sum_i u_i })$ be the linear embedding given by
\[
[x_0:\ldots:x_n]\mapsto[\mathrm{diag}(\underbrace{x_0,\ldots,x_0}_{ u_0\textrm{ times}},\ldots,\underbrace{x_n,\ldots,x_n}_{ u_n\textrm{ times}})],
\]
where the latter denotes the projective equivalence class of the diagonal matrix with the given entries. The determinant function $\det$ is a homogeneous polynomial on $\mathbb{P}^{\sum_iu_i-1}$ such that $(\det\circ\,\varphi)(x_0,\ldots,x_n) = \prod_{i=0}^nx_i^{u_i}$. Let $\eta \in (\mathbb{P}^n)^\vee$ be generic. We can always choose $\lambda \in \mathbb{P}(\mathbb{S}^{n+1})^\vee$ and $\nu \in \mathbb{P}(\mathbb{S}^{\sum_i u_i })^\vee$ such that both $\lambda, \nu$ restrict to $\eta$ on $\mathbb{P}^n$. This gives us the following diagram:
\begin{center}
\begin{tikzpicture}[>=angle 90]
\matrix(a)[matrix of math nodes,
row sep=2em, column sep=2em,
text height=1.5ex, text depth=0.25ex]
{X& \mathbb{P}^n  & \mathbb{P}(\mathbb{S}^{n+1}) & \mathbb{C}\\
&&\mathbb{P}(\mathbb{S}^{\sum_i u_i }). & \\ };
\path[->] (a-1-1) edge node[above]{}(a-1-2);
\path[->] (a-1-2) edge node[above]{}(a-1-3);
\path[->] (a-1-2) edge node[below left]{$\varphi$}(a-2-3);
\path[dashed,->] (a-1-3) edge node[above]{$f_{\det, \lambda}$}(a-1-4);
\path[dashed, ->] (a-2-3) edge node[below right]{$f_{\det, \nu}$}(a-1-4);
\end{tikzpicture}
\end{center}
Although the diagram does not commute, by construction the rational functions $f_{\det, \lambda}, f_{\det, \nu} $ have the same associated reduced divisor $D_{(\prod_i x_i), \eta}$ on $\mathbb{P}^n$ and therefore on $X$.  Theorem~\ref{thm:crit-degree-as-Euler-characteristic} implies then that the number of critical points of both functions is determined by the local Euler obstruction on $X \setminus D_{(\prod_i x_i), \eta}$ and that these points are reduced. We conclude that
\[
\mathrm{Gaussian~ML~degree}(X \subseteq \mathbb{P}(\mathbb{S}^{n+1})) = \mathrm{Gaussian~ML~degree}(X \subseteq \mathbb{P}(\mathbb{S}^{\sum_i u_i})).
\]
 Consider the incidence variety in Remark~\ref{rem:critical-fiber} where we let $F := \prod_i x_i^{u_i}$ and $X \subseteq \mathbb{P}^n$:
\begin{center}
\begin{tikzpicture}[>=angle 90]
\matrix(a)[matrix of math nodes,
row sep=2em, column sep=2em,
text height=1.5ex, text depth=0.25ex]
{&\mathcal{I}&\\
C_X^\circ&&\mathbb{C}^{n+1}.\\};
\path[->] (a-1-2) edge node[above left]{$\pi_1$}(a-2-1);
\path[->] (a-1-2) edge node[above right]{$\pi_2$}(a-2-3);
\end{tikzpicture}
\end{center}
The above argument shows that $\deg(\pi_2) = \mathrm{Gaussian~ML~degree}(X)$.
Now consider the special fiber $\pi_2^{-1}(1,\ldots, 1)$ in this diagram. This fiber is the scheme of critical points of $\ell_{\prod_i x_i^{u_i}, (1,\ldots, 1)}|_{C_X^\circ}$ which in turn by Theorem~\ref{Gaussian-rational-schemes-are-equal} is the scheme of critical points of the rational function in \eqref{eq:disc-likelihood-function}. This is a reduced zero scheme and the count of its points is the discrete ML degree of $X$. There is an open subset $V \subseteq \mathbb{C}^{n+1}$, now containing $(1,1, \ldots, 1)$, such that $\pi_2|_{\pi_2^{-1}(V)}$ is quasifinite (all fibers are finite), see for example \cite[Theorem~10.9(c)]{Mil09}. By the Zariski Main Theorem the map $\pi_2|_{\pi_2^{-1}(V)}$ can be factored as
$$\pi_2^{-1}(V) \overset{i}{\hookrightarrow} W \overset{\pi}{\to} V,$$
where $i$ is an open birational immersion and $\pi$ is a finite map. Because $i$ is birational, $\deg(\pi) = \deg(\pi_2|_{\pi_2^{-1}(V)}) = \deg(\pi_2)$. Because the fibers are finite it follows that: 
\begin{align*}
    \mathrm{Discrete~ML~degree}(X) &= \deg(\pi_2^{-1}(1, \ldots, 1)) =  \#\,\pi_2^{-1}(1, \ldots, 1) \leq \#\,\pi^{-1}(1, \ldots, 1)  \\
    &\leq \deg(\pi) = \deg(\pi_2) = \mathrm{Gaussian~ML~degree}(X),
\end{align*}
where the last inequality follows by \cite[thm~10.12]{Mil09} and recall ``\#'' denotes the set-theoretic count of points. Note that the nature of the first inequality comes from the fact that, to make $W\rightarrow V$ proper, possibly some points were added over $(1,\ldots,1)$ (as demonstrated in Example~\ref{ex: discLMLleqgaussML}). This concludes the proof.
\end{proof}

\begin{example}
\label{ex: discLMLleqgaussML}
Consider the smooth curve $C = V(4x_0x_2 - x_1^2) \subseteq \mathbb{P}^2$. Its Euler characteristic is $2$ because it is a smooth plane conic, hence it is isomorphic to $\mathbb{P}^{1}$. Consider the line cut out by $x_0 + x_1 + x_2$, it is tangent to $C$ in the point $[1:-2:1]$. This gives that the discrete ML degree equals $-(2-(2+1))=1$. However, when considering the embedding of $\mathbb{P}^2$ as the diagonal matrices in $\mathbb{P}(\mathbb{S}^3)$ then the Gaussian ML degree is computed with a generic line which instead intersects $C$ in two points. This gives Gaussian ML degree $-(2-(2+2))=2$.
\end{example}


\newcommand{\etalchar}[1]{$^{#1}$}

\end{document}